\documentclass[a4pape,preprint]{amsart}
\usepackage{a4}
\usepackage{amsmath}
\usepackage{amssymb}
\usepackage[applemac]{inputenc}
\usepackage{graphicx}
\usepackage{color}
\usepackage{subcaption}
\usepackage{diagbox}
\usepackage{comment}
\usepackage{algorithm}
\usepackage[noend]{algpseudocode}
\def\a{\alpha}

\newcommand{\M}{{\mathbb M}}

\newcommand{\ds}{\displaystyle}

\newcommand{\prox}{\ensuremath{\operatorname{prox}}}
\newcommand{\RX}{\ensuremath{\left]-\infty,+\infty\right]}}
\newcommand{\Id}{\mathrm{Id}}

\newcommand{\TT}{{\mathbb{T}}}

\newtheorem{remark}{\textbf{Remark}}[section]

\newtheorem{lemma}{\textbf{Lemma}}[section]
\newtheorem{theorem}{\textbf{Theorem}}[section]

\numberwithin{equation}{section}

\def\dd{{\rm d}}

\def\weight(#1,#2){c_{#1,#2}}

 \if {
  
 } \fi

\if {

}{{\caption}}
} \fi

\def\B{\mathcal{B}}

\def\F{\mathcal{F}}
\def\G{\mathcal{G}}

\def\L{\mathcal{L}}
\def\M{\mathcal{M}}

\def\T{\mathcal{T}}
\def\U{\mathcal{U}}

\def\W{\mathcal{W}}

\def\1B{{\bf  1}}

\newcommand{\NN}{\mathbb{N}}
\newcommand{\ZZ}{\mathbb{Z}}

\newcommand{\RR}{\mathbb{R}}

\newcommand\be{\begin{equation}}
\newcommand\ee{\end{equation}}
\newcommand\ba{\begin{array}}
\newcommand\ea{\end{array}}
\newcommand{\bean}{\begin{eqnarray*}}
\newcommand{\eean}{\end{eqnarray*}}

\def\ds{\displaystyle}

\title[A primal-dual algorithm for  time-dependent MFGs ]{On the implementation of a primal-dual algorithm for second order time-dependent  mean field games with local couplings}

\author{
L. Brice\~no-Arias}\thanks{Universidad T\'ecnica Federico Santa
Mar\'ia, Departamento de Matem\'atica, Av. Vicu\~na Mackenna 3939,
San Joaqu\'in, Santiago, Chile. luis.briceno@usm.cl}
\author{D. Kalise}\thanks{Department of Mathematics, Imperial College London, South Kensington Campus, London SW7 2AZ, United Kindgom. dkaliseb@ic.ac.uk}
\author{Z. Kobeissi}\thanks{Laboratoire Jacques-Louis Lions, Univ. Paris Diderot, Sorbonne Paris Cit\'e, UMR 7598, UPMC, CNRS, 75205, Paris, France. zkobeissi@math.univ-paris-diderot.fr}
\author{M. Lauri\`ere}\thanks{ORFE, Princeton University, Princeton, NJ 08540, USA. lauriere@princeton.edu}
\author{\'A. Mateos Gonz\'alez}\thanks{
Institut Montpellli\'erain Alexander Grothendieck (IMAG), UMR CNRS 5149, Universit\'e de Montpellier, 34090 Montpellier, France,
and Institut des Sciences de l'\'Evolution de Montpellier (ISEM), UMR CNRS 5554, Universit\'e de Montpellier, 34095 Montpellier, France,
and MISTEA, UMR CNRS 0729, INRA and SupAgro Montpellier, 34060 Montpellier, France.
alvaro.mateos-gonzalez@umontpellier.fr}
\author{F. J. Silva}\thanks{Toulouse School of Economics, Universit\'e de Toulouse I Capitole, 31015 Toulouse, France and Institut de recherche XLIM-DMI, UMR-CNRS 7252 Facult\'e
des sciences et techniques 
Universit\'e de Limoges, 87060 Limoges, France. \\
francisco.silva@unilim.fr }

\begin{document}

\maketitle
\begin{abstract}
We study a numerical approximation of a time-dependent Mean Field Game (MFG) system with local couplings. The discretization we consider stems from a variational approach described in~\cite{BAKS} for the stationary problem and leads to the finite difference scheme introduced by Achdou and Capuzzo-Dolcetta in \cite{MR2679575}. In order to solve the finite dimensional variational problems, in \cite{BAKS}  the authors implement the primal-dual algorithm introduced by Chambolle and Pock in \cite{MR2782122}, whose core consists in iteratively solving linear systems and applying a proximity operator. We apply that method to time-dependent MFG and, for large viscosity parameters, we improve the linear system solution  by replacing the direct approach used in~\cite{BAKS} by suitable preconditioned iterative algorithms.
\end{abstract}
\section{Introduction} In this work we consider the following MFG system with local couplings   
\begin{equation}
\label{MFGcont}
\left\{
\begin{aligned}
&-\partial_t u -\nu\Delta u + H(x, \nabla u) =f(x,m(x,t)) &\text{ in }\mathbb{T}^d \times [0,T],\\
&\partial_t m -\nu\Delta m - \mbox{div}(\nabla_p H(x,Du)m)=0 &\text{ in } \mathbb{T}^d\times [0,T],\\
&m(\cdot,\cdot)=m_0(\cdot), \hspace{0.3cm} u(\cdot,T)=g(\cdot,m(\cdot,T)) &\text{ in } \mathbb{T}^d.
\end{aligned}
\right.\tag{MFG}
\end{equation}
In the notation above  $\nu \geq 0$, $d\in \NN$, $\TT^d$ is the $d$-dimensional torus,  $H: \TT^d \times \RR^d \to \RR$ is jointly continuous and convex with respect to its second variable, $f$, $g: \TT^d\times \RR \to \RR$ are continuous functions and $m_0 \in L^1(\TT^d)$ satisfies $m_0 \geq 0$ and $\int_{\TT^d}m_0(x) \dd x=1$. 

System \eqref{MFGcont} has been introduced by J.-M. Lasry and P.-L. Lions in \cite{LasryLions06ii, LasryLions07} in order to describe the asymptotic behaviour of symmetric stochastic differential games as the number of players tends to infinity. Several analytical techniques can be used to prove the existence of solutions to \eqref{MFGcont} under various assumptions on the data.  Despite the recent introduction of the MFG system, the literature dedicated to its theoretical study is already too rich  to be covered exhaustively in this introduction. The interested reader may refer to the 
monographs  \cite{MR3134900,MR3559742},  the surveys \cite{Cardialaguet10,MR3195844} and the references therein for the state of the art of the subject.

A useful approach that can be used to establish the existence of solutions to  \eqref{MFGcont} is the variational one, already presented in \cite{LasryLions07}. The main idea behind is that, at least formally, system  \eqref{MFGcont} can be seen as the first order optimality condition associated to minimizers of the following optimization problem
\be\label{continuous_optimization_problem}\ba{l}
\hspace{0.6cm} \inf_{(m,w)} \hspace{0.3cm} \int_{0}^{T}\int_{\TT^d} \left[ b(x,m(x,t), w(x,t)) + F(x,m(x,t)) \right] \dd x + \int_{\TT^d} G(x,m(x,T)) \dd x \\[6pt]
\mbox{subject to } \hspace{0.6cm} \partial_t m -\nu \Delta m + \mbox{div}(w) =0 \hspace{0.3cm} \mbox{in } \; \TT^d \times (0,T), \\[6pt]
\hspace{2.3cm} m(\cdot,0) = m_0(\cdot) \hspace{0.3cm} \mbox{in } \TT^d,
\ea \tag{P}
\ee
(provided that they exist). In \eqref{continuous_optimization_problem}, the functions $b:\TT^d \times \RR \times \RR^d \to \RR \cup \{+\infty\}$ and $F$, $G: \TT^d \times \RR \to \RR \cup \{+\infty\}$ are defined as follows
\be\label{functions_definitions} \ba{c}
b(x,m,w) := \left\{
\ba{ll}
 mH^*(x,-\frac{w}{m}) 
&\text{ if } m>0,\\
 0 &\text{ if } (m,w)=(0,0), \\
 +\infty &\text{ otherwise,}
\ea
\right.\\[25pt] 

F(x,m) := \left\{ \ba{ll} \int_{0}^{m} f(x,m') \dd m' & \mbox{if } m\geq 0, \\[4pt]
							+\infty & \mbox{otherwise,}\ea \right. \hspace{0.3cm} G(x,m) := \left\{ \ba{ll} \int_{0}^{m} g(x,m') \dd m' & \mbox{if } m\geq 0, \\[4pt]
							+\infty & \mbox{otherwise,}\ea \right.
\ea  
\ee
where, in the definition of $b$,  $H^{\ast}(x,\cdot)$ denotes the Legendre-Fenchel conjugate of $H(x,\cdot)$. Under the assumption that $f(x,\cdot)$ and $g(x,\cdot)$ are non-decreasing, problem $(P)$ is shown to be a convex optimization problem and convex  duality techniques can be successfully applied in order to provide existence and uniqueness results to \eqref{MFGcont}. This argument has been made rigorous in several articles: let us mention \cite{MR3408214,MR3358627}  in the context of first order MFGs ($\nu=0)$, the paper \cite{MR3399179} for degenerate second order MFGs, and finally  \cite{MR3420414,meszaros_silva_preprint_17}  for ergodic second order MFGs.  

The variational approach described above has also been successful in the numerical resolution of system  \eqref{MFGcont}. In this direction, we mention the article \cite{MR2647032} dealing with applications in economics, the paper \cite{MR2888257} concerned with the so-called planning problem in MFGs, the works \cite{MR3395203,andreev_17} focused on the resolution of a discretization of \eqref{continuous_optimization_problem} by the Alternating Direction Method of Multipliers (ADMM) and \cite{BAKS} where several first order methods are implemented and compared for the stationary version of \eqref{MFGcont}. Let us mention that the variational approach is closely related to the so-called mean field optimal control problem, for which numerical methods have been studied in \cite{BFMW14,ACFK17}, among others.

In this paper we consider a finite difference discretization of problem  \eqref{continuous_optimization_problem}. Assuming that  $f(x,\cdot)$ and $g(x,\cdot)$ are non-decreasing, the discretization that we consider is such that it preserves the convexity properties of problem \eqref{continuous_optimization_problem} and the first order optimality conditions for its solutions, which are shown to exist,  coincide  with the finite difference scheme for MFGs introduced in \cite{MR2679575}. A very nice feature of this approach is that the solutions of the resulting discretized MFGs are shown to converge to the solutions of \eqref{MFGcont}. We refer the reader to \cite{MR3097034}, where the convergence result is obtained under the assumption  that  \eqref{MFGcont} admits a unique classical solution, and to \cite{MR3452251} in the framework of weak solutions (see \cite{MR3305653} for the definition of this notion).  We solve the discrete convex optimization problem by using the primal-dual algorithm introduced in \cite{MR2782122}. As was pointed out in \cite{BAKS} (see also \cite{MR3158785} in the context of transport problems), the primal-dual algorithm we consider seems to be faster than the ADMM when $\nu$ in \eqref{MFGcont} is small (or null). On the other hand, the efficiency of both methods is arguable when $\nu$ is large. This is due to the fact that, in both algorithms, at each iteration  one has to invert a matrix whose condition number importantly increases as the  viscosity parameter increases. Naturally, preconditioning strategies (see e.g. \cite{Benzi2002}) can then be used in order to improve the efficiency of both algorithms. This strategy has been already successfully implemented in \cite{andreev_17} for the ADMM.

Our main objective in the present work is to take a closer look at the phenomenon described at the end of previous paragraph when considering the primal-dual algorithm. Therefore, we focus our analysis in the case where $\nu>0$. We have implemented standard indirect methods for solving the linear systems appearing in the computation of the iterates of the primal-dual algorithm. As our numerical results suggest, it is very important to design suitable preconditioning strategies in order to be able to find solutions of the discretization of problem $(P)$ efficiently, and in a robust way with respect to the viscosity parameter. For this, we explore different preconditioning strategies, and in particular, multigrid preconditioning (see also \cite{AchdouPerez2012,andreev_17}, where multigrid strategies have been implemented for other solution methods). 

The article is organized as follows. In section \ref{preliminaries} we introduce some standard notation and we recall the finite difference scheme for \eqref{MFGcont} introduced in \cite{MR2679575}. The variational interpretation of this finite difference scheme is discussed in section \ref{variational_interpretation}. Next, in section \ref{chambolle_pock_section}, we recall the primal-dual algorithm introduced in \cite{MR2782122} and we consider its application to the discretization of \eqref{continuous_optimization_problem}. In section \ref{sec_precond}, we summarize the preconditioning strategies that we consider and we discuss a numerical example, which is the time-dependent version of one of the examples treated in \cite{MR2679575,BAKS}.

\section{Preliminaries and the finite difference scheme in \cite{MR2679575}}\label{preliminaries}

In this section we introduce some basic notation and present the finite difference scheme introduced in \cite{MR2679575}, whose efficient resolution will be the main subject of this article. For the sake of simplicity, we will assume that $d=2$ and that given $q >1$, with conjugate exponent denoted by $q'= q/(q-1)$, the Hamiltonian $H: \TT^2 \times \RR^2 \to \RR$ has the form 
$$
H(x,p)= \frac{1}{q'} |p|^{q'} \hspace{0.3cm} \forall \; x\in \TT^2, \; p \in \RR^2. 
$$
In this case the function $b$ defined in \eqref{functions_definitions} takes the form
$$
b(x,m,w)= \left\{
\ba{ll} \frac{|w|^q}{qm^{q-1}}
&\text{ if } m>0,\\
 0 &\text{ if } (m,w)=(0,0), \\
 +\infty &\text{ otherwise.}
\ea
\right.
$$
Let $N_T$, $N_h$ be positive integers and set $\Delta t=T/N_T$, the time step, and $h=1/N_h$, the space step. We associate to these steps a time grid $\T_{\Delta t}:= \{t_k = k\Delta t \; ; \; k=0,\dots,N_T\}$ 
and   a space grid    $\mathbb{T}^2_h:= \{x_{i,j}=(ih,jh) \; ; \;  i, \; j \in \ZZ\}$. Since $\mathbb{T}^2_h$ intends to discretize  $\mathbb{T}^2$, we impose the identification $z_{i,j}= z_{(i \; \mbox{ \small mod} \; N_h), (j \; \mbox{ \small mod} \; N_h)}$, which allows to assume that $i$, $j \in \{0,\hdots, N_{h}-1\}$.  A function $y:= \TT^2 \times [0,T] \to \RR$ is approximated by its values at  $(x_{i,j},t_k)\in \mathbb{T}^2_h\times \T_{\Delta t}$, which we denote by $y^k_{i,j}:=y(x_{i,j},t_k)$.
Given $y:\mathbb{T}^2_h \to \mathbb{R}$ we define the first order finite difference operators

\begin{equation}
\begin{aligned}
&(D_1y)_{i,j} := \frac{y_{i+1,j}-y_{i,j}}{h}, \text{ and } (D_2y)_{i,j} := \frac{y_{i,j+1}-y_{i,j}}{h}, 
\\
&[D_hy]_{i,j} := ((D_1y)_{i,j},(D_1y)_{i-1,j}, (D_2y)_{i,j}, (D_2 y)_{i,j-1}),
\\
&\widehat{[D_hy]}_{i,j}=((D_1y)^-_{i,j},-(D_1y)^+_{i-1,j}, (D_2y)^-_{i,j}, -(D_2 y)^+_{i,j-1}),
\end{aligned}
\end{equation}
where, for every $a\in \mathbb{R}$, we set $a^+:=\max(a,0)$ and $a^-:=a^+-a$. The discrete Laplacian operator $\Delta_hy:\mathbb{T}^2_h \rightarrow \mathbb{R}$ is defined by 
$$
(\Delta_hy)_{i,j} := -\frac{1}{h^2}(4y_{i,j}-y_{i+1,j}-y_{i-1,j} -y_{i,j+1} - y_{i,j-1}).
$$
For $y:  \T_{\Delta t} \to  \mathbb{R}$ we define the discrete time derivative
$$
 D_ty^k := \frac{y^{k+1}-y^k}{\Delta t}.
$$
The Godunov-type finite difference discretization of \eqref{MFGcont} introduced in  \cite{MR2679575} is as follows: find $u$, $m: \TT_h^2 \times \T_{\Delta t} \to \RR$ such that 
 for all    $0\leq i,j\leq N_h-1$ and $0\leq k \leq N_T-1$ we have
\begin{equation}
\label{MFGdis}
\left\{
\begin{array}{l}
\hspace{0.3cm} -D_t u_{i,j}^k -\nu(\Delta_h u^k)_{i,j} + \frac{1}{q'}|\widehat{[D_hu^k]}_{i,j}|^{q'}  = f(x_{i,j},m_{i,j}^{k+1}), \\[10pt]
  D_t m_{i,j}^k -\nu(\Delta_h m^{k+1})_{i,j} - \mathcal{T}_{i,j}(u^k,m^{k+1}) = 0, \\[10pt]
\hspace{0.3cm} m_{i,j}^0=\bar{m}_{i,j},   \hspace{0.6cm}  u_{i,j}^{N_T}=g(x_{i,j},m_{i,j}^{N_T}),
\end{array}\right.   \tag{MFG$_{h,\Delta t}$}
\end{equation}
where 
\be\label{initial_condition_discretization}
\bar{m}_{i,j}:= \int_{|x-x_{i,j}|_{\infty} \leq \frac{h}{2}} m_0(x) \dd x \geq 0,
\ee
and the operator $\mathcal{T}(u',m'):\mathbb{T}^2_h \rightarrow \mathbb{R}$, with $u',m':\mathbb{T}^2_h \rightarrow \mathbb{R}$, is defined by 

$$\ba{ll}
 \mathcal{T}_{i,j}(u',m'):=& \frac{1}{h} \left( 
-m'_{i,j}\frac{1}{q'}|\widehat{[D_hu']}_{i,j}|^{\frac{2-q}{q-1}}(D_1u')_{i,j}^-
+m'_{i-1,j}\frac{1}{q'}|\widehat{[D_hu']}_{i-1,j}|^{\frac{2-q}{q-1}}(D_1u')_{i-1,j}^- \right.\\[10pt]
\; & \hspace{0.6cm} +m'_{i+1,j}\frac{1}{q'}|\widehat{[D_hu']}_{i+1,j}|^{\frac{2-q}{q-1}}(D_1u')_{i,j}^+
-m'_{i,j}\frac{1}{q'}|\widehat{[D_hu']}_{i,j}|^{\frac{2-q}{q-1}}(D_1u')_{i-1,j}^+\\[10pt]
\; & \hspace{0.6cm} -m'_{i,j}\frac{1}{q'}|\widehat{[D_hu']}_{i,j}|^{\frac{2-q}{q-1}}(D_2u')_{i,j}^-
+m'_{i,j-1}\frac{1}{q'}|\widehat{[D_hu']}_{i,j-1}|^{\frac{2-q}{q-1}}(D_2u')_{i,j-1}^-\\[10pt]
\; & \hspace{0.6cm} \left.+m'_{i,j+1}\frac{1}{q'}|\widehat{[D_hu']}_{i,j+1}|^{\frac{2-q}{q-1}}(D_2u')_{i,j}^+
-m'_{i,j}\frac{1}{q'}|\widehat{[D_hu']}_{i,j}|^{\frac{2-q}{q-1}}(D_2u')_{i,j-1}^+ \right),
\ea
$$
\normalsize
with the convention:

\begin{equation}\label{convention}
|\widehat{[D_hu']}_{i,j}|^{\frac{2-q}{q-1}}\widehat{[D_hu']}_{i,j}=0
\text{ if } q>0 \text{ and } \widehat{[D_hu']}_{i,j}=0.
\end{equation}


The existence of a solution $(u^{h,\Delta t}, m^{h,\Delta t})$ of system \eqref{MFGdis} is proved in \cite[Theorem 6]{MR2679575} as a consequence of Brouwer fixed point theorem. Furthermore, if we assume that $f$ and $g$ are  increasing with respect to their second argument, and one of them is strictly increasing, this solution is unique when $h$ is small enough (see \cite[Theorem 7]{MR2679575}). As we will see in the next section, these results can also be obtained by variational arguments. The convergence,  as $h$ and $\Delta t$ tend to $0$, of suitable extensions of $u^{h,\Delta t}$ and $m^{h, \Delta t}$ to $\TT^2 \times [0,T]$ to a solution $(u,m)$ of \eqref{MFGcont} is proved in \cite{MR3097034} under the assumption that $(u,m)$ is unique and sufficiently regular. The later smoothness assumption has been relaxed in \cite{MR3452251}. 

%


\section{The finite dimensional variational problem and the discrete MFG system}\label{variational_interpretation}

Following \cite{BAKS} in the stationary case and \cite{MR2888257} for the planning problem, we introduce some finite-dimensional operators that will allow us to write easily a finite dimensional version of problem \eqref{continuous_optimization_problem}. Denoting by $\mathbb{R}_+$ the set of non-negative real numbers  and by $\mathbb{R}_-$ the set of non-positive real numbers, we define $K:=\mathbb{R}_+\times\mathbb{R}_-\times\mathbb{R}_+\times\mathbb{R}_-$  and for $v = (v^{(1)},v^{(2)},v^{(3)},v^{(4)}) \in \RR^4$ we denote by $P_K(v) = ((v^{(1)})^+,-(v^{(2)})^-,(v^{(3)})^+,-(v^{(4)})^-) $ its orthogonal projection  onto $K$.                                                 
Let $\M:=\mathbb{R}^{(N_T+1)\times N_h\times N_h}$,  
$\mathcal{W}:=(\mathbb{R}^4)^{N_T\times N_h \times N_h}$ and  $\U:=  \mathbb{R}^{N_T\times N_h\times N_h}$.  Let  $A: \M \to \U$ and $B:\mathcal{W}\to \U$ be the linear operators defined by  
\begin{equation}\label{definition_of_A_B}
\begin{aligned}
&(Am)_{i,j}^k :=  D_tm_{i,j}^{k}-\nu(\Delta_hm^{k+1})_{i,j},  \\
&(Bw)_{i,j}^k :=  (D_1w^{k,(1)})_{i-1,j}+(D_1w^{k,(2)})_{i,j}+(D_2w^{k,(3)})_{i,j-1}+(D_2w^{k,(4)})_{i,j},
\end{aligned}
\end{equation}
for all $0\leq i,j \leq N_h-1$ and $0\leq k\leq N_T-1$. One can easily check (see e.g. \cite{MR2679575}) that the corresponding dual  operators are given by
\begin{equation}\label{adjoints_operators}
\begin{aligned}
&(B^*u)_{i,j}^k= -[D_hu^k]_{i,j} \quad \text{ for all } 0\leq k \leq N_T-1,\\
&(A^*u)_{i,j}^k= - D_tu_{i,j}^{k-1}-\nu(\Delta_hu^{k-1})_{i,j}, \quad \text{ if } 1\leq k \leq N_T-1, \\
&(A^*u)_{i,j}^0= -\frac{1}{\Delta t}u_{i,j}^0, \\
&(A^*u)_{i,j}^{N_T}= \frac{1}{\Delta t}u_{i,j}^{N_T-1}-\nu(\Delta_hu^{N_T-1})_{i,j},
\end{aligned}
\end{equation}
for all $u \in \U$. For later 
 use,  notice that 
$$\mbox{Ker}(B^\ast)= \{ u \in \U \; | \;   \forall \; k=0, \hdots, N_T-1 \; \; \mbox{there exists } c_{k} \in \RR \; \; \mbox{such that } u_{i,j}^k=c_k \; \; \forall \; i, j\},$$
and so 
\be\label{image_of_B}
\mbox{Im}(B)= \mbox{Ker}(B^\ast)^{\perp}= \Big\{ u \in \U \; \big| \;  \sum_{i,j} u_{i,j}^k= 0 
\hspace{0.3cm} \forall \; k=0, \hdots, N_T-1\Big\}.
\ee
Let us define $\widehat{b}:\mathbb{R}\times\mathbb{R}^4 \to   \RR\cup \{+\infty\}$
\be\label{definition_b_hat}
\widehat{b}(m,w):=
\left\{
\ba{ll}
  \frac{|w|^q}{qm^{q-1}},
&\text{ if } m>0, w \in K,\\[6pt]
 0, &    \text{ if } (m,w)=(0,0), \\[6pt]
 +\infty, &\text{ otherwise},
\ea
\right.
\ee
and the functions  $\B$, $\F: \M \times \W \to \RR$, $\G: \M \times \W \to \M \times \RR^{N_h \times N_h}$ as

\begin{equation}\label{definitions_functions_optimization_problem}
\ba{rcl}
\B(m,w)&:=& \ds \sum_{\substack{1\leq k \leq N_T, \\ 0\leq i,j\leq N_h-1}}\widehat{b}(m_{i,j}^k,w_{i,j}^{k-1}),\\[10pt]
\F(m)&:=&  \ds \sum_{\substack{1\leq k \leq N_T, \\
0\leq i,j\leq N_h-1}}F(x_{i,j},m_{i,j}^k)
+ \frac{1}{\Delta t}\sum_{0\leq i,j\leq N_h-1} G(x_{i,j},m_{i,j}^{N_T}),\\[25pt]
\G(m,w)&:=&(Am+Bw,m^0).
\ea
\end{equation}
Note that if $(m,w) \in \M \times \W$ is such that $\G(m,w)=(0,\bar{m})$, where we recall that $\bar{m}$ is defined in \eqref{initial_condition_discretization}, then 
\be\label{conservation_of_mass_property}
h^2 \sum_{i,j} m_{i,j}^k = 1 \hspace{0.3cm} \forall \; k=0,\hdots, N_T. 
\ee
Indeed, by periodicity, $-\sum_{i,j} (\Delta_h m^{k+1})_{i,j}=0$ and $\sum_{i,j} (Bw)_{i,j}^{k}=0$ for all $k=0, \hdots, N_T-1$. This implies that 
$$
0 = \sum_{i,j} (Am+Bw)^k_{i,j}= \frac{\sum_{i,j}m^{k+1}_{i,j}}{\Delta t} -\frac{\sum_{i,j}m^{k}_{i,j}}{\Delta t}, 
$$
and so $h^2\sum_{i,j}m^{k}_{i,j}= h^2\sum_{i,j}\bar{m}_{i,j}=1$ for all $k=0, \hdots, N_T$. 

%
The discretization of the variational problem \eqref{continuous_optimization_problem} that we consider is 
\begin{equation}
\label{VARdis}
\inf_{(m,w)\in \mathcal{M}\times\mathcal{W}}
\mathcal{B}(m,w)+\mathcal{F}(m), \; \;  \text{ subject to } \; \;  \G(m,w)=(0,\bar{m}), \tag{P$_{h,\Delta t}$}
\end{equation}
where we recall that $F$ and $G$ in \eqref{definitions_functions_optimization_problem} are defined in \eqref{functions_definitions}. 

We have the following result
\begin{theorem}
\label{thm:existence}
For any $\nu > 0$ problem  \eqref{VARdis} admits at least one solution $(m^{h,\Delta t}, w^{h,\Delta t})$ and associated to it there exists $u^{h,\Delta t}: \M  \times \W  \to \RR$ such that \eqref{MFGdis} holds true. Moreover, $(m^{h,\Delta t})^{k}_{i,j}>0$ for all $k=1, \hdots, N_T$,  $i$, $j =0, \hdots, N_h -1$. 
%
%
%
\end{theorem}
In order to prove the result above, let us first show a lemma that implies the feasibility of the constraints in \eqref{VARdis}.
\begin{lemma}
\label{lem:feasibility}
There exists $(\tilde{m},\tilde{w}) \in \mathcal{M}\times\mathcal{W}$ such that 

\begin{equation}\ba{l}
\G(\tilde{m},\tilde{w})=(0,\bar{m}), \quad \tilde{w}_{i,j}^k\in \text{{\rm int}}(K)  \hspace{0.3cm}\forall  \; i, \; j =1, \hdots, N_h-1, \; \; k=1, \hdots, N_T-1, \\[6pt]
 \tilde{m}_{i,j}^k>0, \hspace{0.3cm}\forall  \; i, \; j =1, \hdots, N_h-1, \; \; k=1, \hdots, N_T.
\ea
\end{equation}
\end{lemma}
\begin{proof}
Let us define $\tilde{m}^0_{i,j}:=\bar{m}_{i,j}$ and $\tilde{m}^{k}_{i,j}:= 1$ for all $k=1,\hdots, N_T$ and $i$, $j$. Since $h^2\sum_{i,j}\tilde{m}_{i,j}^k=1$ for all  $k=0,\hdots, N_T$, by \eqref{image_of_B} and the definition of $A$ we easily get that  $A\tilde{m} \in  \mbox{Im}(B)$. Therefore,  there exists $\hat{w} \in \mathcal{W}$ satisfying $\G(\tilde{m},\hat{w})=(0,\bar{m})$. Then,  given $\delta>0$, we  set  for all $k=0, \hdots, N_T-1$ and $i$, $j$  \small
$$
\tilde{w}^{k}_{i,j}:= \left(\hat{w}^{k,(1)}+\max_{i,j} \hat{w}^{k,(1)}_{i,j} +\delta,\hat{w}^{k,(2)}-\max_{i,j} \hat{w}^{k,(2)}_{i,j} -\delta, \hat{w}^{k,(3)}+\max_{i,j} \hat{w}^{k,(3)}_{i,j} +\delta, \hat{w}^{k,(4)}-\max_{i,j} \hat{w}^{k,(4)}_{i,j} -\delta\right),
$$ \normalsize
which satisfies $\tilde{w}_{i,j}^k \in \mbox{int}(K)$  and  $(B\tilde{w})^k= (B\hat{w})^k$. The result follows.
\end{proof}
Now, we prove the existence of solutions to \eqref{VARdis}. 
\begin{lemma}\label{existence_of_solutions_discrete_variational_problems} Problem \eqref{VARdis} admits at least one solution $(m^{h,\Delta t}, w^{h,\Delta t})$ and every such solution satisfies 
$(m^{h,\Delta t})^{k}_{i,j}>0$ for all $k=1, \hdots, N_T$,  $i$, $j =0, \hdots, N_h -1$.
\end{lemma}
\begin{proof} Let $(m^n, w^n)$ be a minimizing sequence for \eqref{VARdis}. Lemma \ref{lem:feasibility} implies that $\mathcal{B}(\tilde{m},\tilde{w})+\mathcal{F}(\tilde{m}) <+\infty$. Therefore,  there exists a constant $C_1>0$ such that 
\be\label{minimizing_sequence_bounded_cost}\mathcal{B}(m^n,w^n)+\mathcal{F}(m^n) \leq C_1 \hspace{0.3cm} \mbox{for all $n\in \NN$.}\ee
 As a consequence, by definition of $\hat{b}$, $(m^n)_{i,j}^k \geq 0$ for all $i$, $j$ and $k$ and $(w^n)^k \in K$ for all $k$.  Since $Am^n + Bw^n=0$, relation \eqref{conservation_of_mass_property} implies that
$h^2 \sum_{i,j} (m^n)_{i,j}^k = 1$. In particular, there exists $C_2>0$ (independent of $n$) such that $\sup_{i,j,k} (m^n)_{i,j}^k \leq C_2$. Using that, if $(m^n)_{i,j}^k>0$, 
$$
\hat{b}((m^n)_{i,j}^k,(w^n)_{i,j}^k) \geq \frac{|(w^n)_{i,j}^k|^q}{q C_2^{q-1}},
$$
and that $\F(m^n)$ is uniformly bounded (because $F$ and $G$ are continuous and $m^n$ is bounded), relation \eqref{minimizing_sequence_bounded_cost} yields the existence of $C_3>0$  (independent of $n$) such that $\sup_{i,j,k} |(w^n)_{i,j}^k| \leq C_3$. Thus, there exists $(m^{h, \Delta t},w^{h, \Delta t}) \in \M \times \W$ such that, up to some subsequence, $m^n \to m^{h, \Delta t}$ and $w^n \to w^{h, \Delta t}$ as $n\to \infty$. Since $\G(m^n,w^n)=(0,\bar{m})$ we obtain that $\G(m^{h, \Delta t},w^{h, \Delta t})=(0,\bar{m})$, The lower semicontinuity of $\B + \F$ implies that 
$$ 
\mathcal{B}(m^{h, \Delta t},w^{h, \Delta t})+\mathcal{F}(m^{h, \Delta t}) \leq \lim_{n\to \infty}\mathcal{B}(m^n,w^n)+\mathcal{F}(m^n),
$$
which implies that $(m^{h, \Delta t},w^{h, \Delta t})$ solves \eqref{VARdis}. Finally, if $(m,w)\in \M \times \W$ solves \eqref{VARdis} and $m^{k}_{i,j}=0$ for some $i$, $j$ and $k=1, \hdots, N_T$, then, by the definition of $\B$, we must have that $w^{k-1}_{i,j}=0$. Thus, the constraint $(Am+Bw)^{k-1}_{i,j}=0$ can be written as 
$$\ba{c}
-\frac{m_{i,j}^{k-1}}{\Delta t} - \frac{\nu}{h^2}( m^{k}_{i+1,j}+m^{k}_{i-1,j}+m^{k}_{i,j+1} + m^{k}_{i,j-1})\\[8pt]
= \frac{w^{k-1,(1)}_{i-1,j}}{h} - \frac{w^{k-1,(2)}_{i+1,j}}{h} +\frac{w^{k-1,(3)}_{i,j-1}}{h} - \frac{w^{k-1,(4)}_{i,j+1}}{h}.
\ea
$$
Since the left hand side above is non-positive and the right hand side is non-negative (by definition of $K$), we deduce that all the terms above are zero.  By repeating the argument at the indexes neighboring $(i,j)$, we deduce that $m^k\equiv 0$ and so $h^2 \sum_{i,j}m_{i,j}^k=0$ which, by \eqref{conservation_of_mass_property},  contradicts $\G(m,w)=(0,\bar{m})$. The result follows. 
\end{proof}
\begin{remark}
Notice that the proof of the existence of a solution to $(P_{h,\Delta t})$ also works when $\nu=0$. 
\end{remark}

%

\begin{proof}[Proof of Theorem \ref{thm:existence}] By Lemma \ref{existence_of_solutions_discrete_variational_problems} we know that there exists a solution $(m^{h,\Delta t}, w^{h,\Delta t})$ to \eqref{VARdis} and $m^{h,\Delta t}_{i,j}>0$ for all $i$, $j$. Thus, in order to conclude it suffices to show the existence of $u^{h,\Delta t}$ such that \eqref{MFGdis} holds true. For notational convenience we will omit the superindexes $h$ and $\Delta t$. Define the {\it Lagrangian} $\L:=\M \times \W \times \U \times \RR^{N_h \times N_h} \to \RR \cup \{+\infty\}$, associated to \eqref{VARdis}, as 
\be\label{e:lagrang}
\ba{rcl}
\L(m,w,u,\lambda)&:=& \mathcal{B}(m,w)+\mathcal{F}(m) - \langle u,Am + Bw\rangle - \langle \lambda, m^0- \bar{m} \rangle \\[6pt]
\; &=& \mathcal{B}(m,w)+\mathcal{F}(m) - \langle A^\ast u,m\rangle  - \langle B^\ast 
u,w\rangle - \langle \lambda, m^0- \bar{m} \rangle.
\ea
\ee
 Note that the linear mapping $ \M \ni m \mapsto (Am,m) \in \U \times  \RR^{N_h \times N_h}$ is invertible as it is shown by its matrix representation (see \eqref{matrix_representation_a_id} in the next section). As a consequence $\G$ is surjective and,  hence, by standard arguments, there exists $(u,\lambda) \in  \U \times  \RR^{N_h \times N_h}$ such that 
\be\label{stationary_condition}
\ba{l}
0=\partial_{m_{i,j}^{k}} \L(m,w,u,\lambda)= - \frac{1}{q'} \frac{|w_{i,j}^{k-1}|^q}{(m_{i,j}^{k})^q}+f(x_{i,j},m_{i,j}^{k}) - [A^{\ast}u]_{i,j}^{k} \hspace{0.6cm} \forall \; k=1,\hdots,N_{T}-1, \; \forall \; i,j, \\[6pt]
0=\partial_{m_{i,j}^{0}} \L(m,w,u,\lambda)= -\lambda_{i,j}-  [A^{\ast}u]_{i,j}^{0} \hspace{0.6cm} \forall \; i,j,\\[6pt]
0 =\partial_{m_{i,j}^{N_T}} \L(m,w,u,\lambda)=- \frac{1}{q'} \frac{|w_{i,j}^{N_T-1}|^q}{(m_{i,j}^{N_T})^q} + f(x_{i,j},m_{i,j}^{N_T})+ \frac{1}{\Delta t} g(x_{i,j}, m_{i,j}^{N_T})- [A^{\ast}u]_{i,j}^{N_T}\hspace{0.6cm} \forall \; i,j,\\[8pt]
0 \in \partial_{w_{i,j}^{k-1}} \L(m,w,u,\lambda)= |w_{i,j}^{k-1}|^{q-2} \frac{w_{i,j}^{k-1}}{(m_{i,j}^k)^{q-1}}-[B^{\ast}u]_{i,j}^{k-1} + N_{K}(w_{i,j}^{k-1}) \hspace{0.3cm} \forall \; k=1,\hdots,N_{T}, \; \forall \; i,j,
\ea\ee
where we have used  definition \eqref{definition_b_hat} and that $m_{i,j}^{k}>0$ for all $k=1, \hdots, N_T$ and all $i$, $j$.  Defining $u_{i,j}^{N_T}:= g(x_{i,j}, m_{i,j}^{N_T})$, by the last relation in \eqref{adjoints_operators}, the third relation in \eqref{stationary_condition} can be rewritten as  
$$-D_t u_{i,j}^{N_{T}-1}-\nu(\Delta_hu^{N_T-1})_{i,j}+\frac{1}{q'} \frac{|w_{i,j}^{N_T-1}|^q}{(m_{i,j}^{N_T})^q}=    f(x_{i,j},m_{i,j}^{N_T}),$$
and hence, by the second relation in \eqref{adjoints_operators} and the first relation in  \eqref{stationary_condition}, we have that 
\be\label{pre_equation_in_m_in_terms_of_w}
-D_t u_{i,j}^{k}-\nu(\Delta_hu^{k})_{i,j}+\frac{1}{q'} \frac{|w_{i,j}^{k}|^q}{(m_{i,j}^{k+1})^q}=    f(x_{i,j},m_{i,j}^{k+1}) \hspace{0.3cm} \forall \; k=0, \hdots, N_{T}-1, \; \;  \forall \; i, \; j.
\ee
The last relation in \eqref{stationary_condition} yields that   for all $k=1, \hdots, N_T$ and all $i$, $j$
$$
\left\{ \ba{ll} \frac{(m_{i,j}^k)^{q-1}}{|w_{i,j}^{k-1}|^{q-2}} [B^{\ast}u]^{k-1}_{i,j} \in w_{i,j}^{k-1} + N_{K}(w_{i,j}^{k-1}) & \; \mbox{if } \; w_{i,j}^{k-1} \neq 0, \\[6pt]
[B^{\ast}u]^{k-1}_{i,j} \in  N_{K}(0)& \; \mbox{if } \; w_{i,j}^{k-1} = 0,
\ea\right.
$$
which, by \eqref{adjoints_operators} and under the convention \eqref{convention}, is equivalent to  
\be\label{expression_for_w} 
w_{i,j}^{k-1}= m_{i,j}^k |P_{K}(-[D_hu]_{i,j}^{k-1})|^{\frac{2-q}{q-1}} P_{K}(-[D_hu]_{i,j}^{k-1})=m_{i,j}^k | \widehat{[D_hu^{k-1}]}_{i,j}|^{\frac{2-q}{q-1}} \widehat{[D_hu^{k-1}]}_{i,j}.
\ee
Shifting the index $k$, the expression above yields 
$$
\frac{1}{q'} \frac{|w_{i,j}^{k}|^q}{(m_{i,j}^{k+1})^q}= \frac{1}{q'}  |\widehat{[D_hu^{k}]}_{i,j}|^{q'} \hspace{0.3cm} \forall \; k=0, \hdots, N_T-1, \; \; \forall \; i, \; j,
$$
which, combined with \eqref{pre_equation_in_m_in_terms_of_w}, implies the first equation in \eqref{MFGdis}. The second equation in \eqref{MFGdis} is a consequence of $Am+Bw=0$ and the fact  that \eqref{expression_for_w} provides the identity 
$$
(Bw)_{i,j}^{k}=- \mathcal{T}_{i,j}(u^k,m^{k+1}) \hspace{0.3cm} \forall \; k=0, \hdots, N_T-1, \; \; \forall \; i, \; j.
$$
The result follows. 
\end{proof}
\begin{remark} {\rm(i)} The proof of the existence of solutions to \eqref{MFGdis} in Theorem \ref{thm:existence} provides an alternative argument   to the one in \cite{MR2679575}, based on Brouwer fixed-point theorem. \smallskip\\
{\rm(ii)}{\rm(Uniqueness)} If $f(x,\cdot)$ and $g(x,\cdot)$ are increasing, with one of them being strictly increasing, then   \eqref{MFGdis} has a unique solution. Indeed, under this assumption,  the cost functional in \eqref{VARdis} is convex w.r.t. $(m,w)$ and strictly convex w.r.t. $m$. It is easy to check that this implies that if $(m_1,w_1)$ and $(m_2,w_2)$ are two solutions of  \eqref{VARdis} then $m_1=m_2$. Using this fact and the definition of $\hat{b}$ {\rm(}see \eqref{definition_b_hat}{\rm)}, we also get that $w_1=w_2$. Thus, under this monotonicity assumption, the solution $(m^{h,\Delta t}, w^{h,\Delta t})$ to  \eqref{VARdis} is unique. Having this result, the uniqueness of $u^{h,\Delta t}$ follows directly from  \cite[Lemma 1]{MR2679575}. 
\end{remark}

\section{A primal-dual algorithm to solve \eqref{VARdis}}\label{chambolle_pock_section}
 
	As discussed in \cite{BAKS}, for solving the optimization problem 
	\begin{equation}
	\label{e:primal}
	\min_{y\in\RR^N}{\varphi(y)+\psi(y)},
	\end{equation}
	and its dual
		\begin{equation}
	\label{e:dual}
	\min_{\sigma\in\RR^N}{\varphi^*(-\sigma)+\psi^*(\sigma)},
	\end{equation}
	where $\varphi\colon\RR^N\to \RR \cup \{+\infty\}$ and $\psi\colon\RR^N\to \RR \cup 
	\{+\infty\}$ are 
	convex l.s.c. proper functions, methods in \cite{M+S,MR2782122,ChenT,ADMM,ADMM2} 
	can be applied with 
	guaranteed
	convergence under mild assumptions. In \cite{BAKS}, devoted to the stationary case, the 
	method proposed in 
	\cite{MR2782122} has the best performance when the viscosity parameter is small or zero.
This method is inspired by the first-order optimality conditions satisfied by a solution 
$(\hat{y},\hat{\sigma})$ to 
\eqref{e:primal}-\eqref{e:dual} under standard 
qualification conditions,
which reads (see \cite[Theorem~8]{MR0211759})
\begin{equation}
\label{e:caract}
\begin{cases}
-\hat{\sigma}\in\partial\varphi(\hat{y})\\
\hat{y}\in\partial\psi^*(\hat{\sigma})
\end{cases}\Leftrightarrow\quad
\begin{cases}
\hat{y}-\tau\hat{\sigma}\in\tau\partial\varphi(\hat{y})+\hat{y}\\
\hat{\sigma}+\gamma\hat{y}\in\gamma\partial\psi^*(\hat{\sigma})+\hat{\sigma}
\end{cases}\Leftrightarrow\quad
\begin{cases}
\prox_{\tau\varphi}(\hat{y}-\tau\hat{\sigma})=\hat{y}\\
\prox_{\gamma\psi^*}(\hat{\sigma}+\gamma\hat{y})=\hat{\sigma},
\end{cases}
\end{equation}
where  $\gamma>0$ and $\tau>0$ are arbitrary and, given a l.s.c. convex proper function 
$\phi\colon\RR^N\to\RX$, 
$$\prox_{\gamma\phi}x := 
\mbox{argmin}_{y\in\RR^N}\left\{\phi(y)+\frac{|y-x|^2}{2\gamma}\right\}= (I+\partial (\gamma \phi))^{-1}(x) \hspace{0.5cm} 
\forall \; x\in \RR^N.$$
Given $\theta\in[0,1]$, $\tau$ and $\gamma$ satisfying $\tau\gamma<1$, and starting 
points
	$(y^0,\tilde{y}^0,\sigma^0)\in\RR^{N}\times\RR^{N}\times\RR^{M}$, 
the iterates $\{(y^k,\sigma^k)\}_{k\in\NN}$ generated by
	\be\label{pasosChamPock}
	\ba{rcl} \sigma^{k+1}&:=& 
	\prox_{\gamma\psi^*}(\sigma^{k}+ \gamma\tilde{y}^{k}),\\[6pt]
	y^{k+1}&:=&
	\prox_{\tau\varphi}(y^{k}-\tau \sigma^{k+1}),\\[6pt]
	\tilde{y}^{k+1}&:=& 
	y^{k+1}+ \theta(y^{k+1}-y^{k})
	\ea\ee
	converge to a primal-dual solution $(\hat{y},\hat{\sigma})$ to 
	\eqref{e:primal}-\eqref{e:dual} (see, e.g., \cite{MR2782122}).
	
In the case under study, the equations of the time-dependent discretization are very similar 
to their stationary counterparts (see \cite{BAKS}). 
Specifically, 
the discrete linear operators $A$ and $B$ defined in \eqref{definition_of_A_B}, by an abuse of 
notation, are represented by 
real matrices $A$ 
and $B$, of dimensions $({N_T \times N_h^2})\times({(N_T + 1) \times N_h^2})$ and $({N_T 
\times N_h^2})\times({N_T \times   4N_h^2}) $, respectively, given by
\begin{equation}
\label{e:defmA}
A:=\begin{pmatrix}
-\frac{1}{\Delta t}\Id_{N_h^2} & \nu L+\frac{1}{\Delta t}\Id_{N_h^2}
& 0& \cdots & 0
\\
0 & -\frac{1}{\Delta t}\Id_{N_h^2} & \nu L+\frac{1}{\Delta t}\Id_{N_h^2} & \ddots & \vdots 
\\
\vdots & \ddots & \ddots & \ddots & 0
\\
0 & \cdots & 0 & -\frac{1}{\Delta t}\Id_{N_h^2} & \nu L+\frac{1}{\Delta t}\Id_{N_h^2}
\end{pmatrix},
\end{equation}
and 
\begin{equation}
\label{e:defmB}
B:=\begin{pmatrix}
M &0&\cdots&0
\\
0& M &\cdots&0\\
\vdots&&\ddots &\vdots\\
0&\cdots&0&M
\end{pmatrix},
\end{equation}
where $L\in\mathcal{M}_{{N_h^2} \, , \; {N_h^2}} (\RR)$ is the matrix that represents $-\Delta_h$ on the 
torus $\mathbb{T}_h^2$ and $M \in \mathcal{M}_{{N_h^2} \, , \; {4 N_h^2}} (\RR)$ is the 
matrix representing the discrete divergence. Denoting by $\tilde{A}$ and $\tilde{B}$ the 
$({(N_T+1) \times N_h^2})\times({(N_T + 1) \times N_h^2})$ and $({(N_T+1) 
	\times N_h^2})\times({N_T \times   4N_h^2}) $ real matrices 
\begin{equation}\label{matrix_representation_a_id}
\tilde{A} := \begin{pmatrix}
\Id_{N_h^2}\!\!&0&\cdots& 0\\
& A \\
\end{pmatrix}\quad\text{and}\quad 
\tilde{B} := \begin{pmatrix}
0&\cdots& 0\\
& B \\
\end{pmatrix},
\end{equation}
the constraint  $\G(m,w)=(0,\bar{m})$ in \eqref{VARdis} can be  rewritten as 
$C(m,w)=(\bar{m},0)$, where $C:=[\tilde{A}\,|\,\tilde{B}]$. 
\begin{remark} 
	\begin{enumerate}
		\item[{\rm(i)}] The matrix $\tilde{A}$ is block lower triangular with invertible diagonal 
		blocks 
		and, hence, it is invertible. Indeed, the first diagonal block $\Id_{N_h^2}$ is obviously 
		invertible and the other blocks, given by $\nu L+\frac{1}{\Delta t}\Id_{N_h^2}$, are also 
		invertible because they are strictly diagonally dominant. 
		\item[{\rm(ii)}]  Since $\tilde{A}$ is invertible, the matrix 
		\begin{equation}
		\label{e:defQ}
			Q:=CC^*=\tilde{A}\tilde{A}^*+\tilde{B}\tilde{B}^*
		\end{equation} is positive definite and, hence, 
		invertible.
	\end{enumerate}

\end{remark}
Therefore,  
\eqref{VARdis} is a 
particular instance of \eqref{e:primal} with 
$$\varphi (m,w):=
\mathcal{B}(m,w)+\mathcal{F}(m), \hspace{0.8cm} \psi 
(m,w):=\iota_{\ker C+\{({m}_f,{w}_f)\}}(m,w),$$ 
where $({m}_f,{w}_f)$
is a feasible vector (provided for instance by Lemma \ref{lem:feasibility}), and   $\iota_{\ker C+\{({m}_f,{w}_f)\}}$ is the function defined as $0$ for all $(m,w)\in \ker C+\{({m}_f,{w}_f)\}$ and $+\infty$, otherwise. 

Since $\prox_{\gamma\psi*}=\Id-\gamma \prox_{\psi/\gamma}\circ(\Id/\gamma)=\Id-\gamma 
\prox_{\psi}\circ(\Id/\gamma)$ (see e.g. \cite[Section 24.2]{MR3616647})
and  
$$\prox_{\psi}\colon(m,w)\mapsto 
(m,w)-C^*Q^{-1}(C(m,w)-{(\bar{m},0)}),$$
we have
$$\prox_{\gamma\psi*}\colon(m,w)\mapsto 
C^*Q^{-1}(C(m,w)-\gamma{(\bar{m},0)}),$$
where $Q$ is defined in \eqref{e:defQ}.
By setting  
$y^0=(m^0,w^0)$, 
 $\tilde{y}^0=(\widetilde{m}^0, \widetilde{w}^0)$, $\sigma^0=(n^0,v^0)\in\RR^{N_T\times 
N_h^2}\times\RR^{N_T(4 N_h^2)}$,
\eqref{pasosChamPock} becomes
\begin{equation} \label{eq_CP}
\left\{
\begin{aligned}
&z^{[l+1]} =-Q^{-1}\left(\tilde{A}(n^{[l]} + \gamma 
\tilde{m}^{[l]})+\tilde{B}(v^{[l]} + \gamma \tilde{w}^{[l]})-\gamma(\bar{m},0)\right),\\
&\begin{pmatrix} n^{[l+1]} \\ v^{[l+1]} \end{pmatrix} =
\begin{pmatrix} \tilde{A}^*z^{[l+1]} \\ \tilde{B}^*z^{[l+1]}\end{pmatrix}, \\
&\begin{pmatrix} m^{[l+1]} \\ w^{[l+1]} \end{pmatrix} =
\prox_{\tau \varphi} \begin{pmatrix} m^{[l]} + \tau n^{[l+1]} \\ w^{[l]} + \tau v^{[l+1]} 
\end{pmatrix}, \\
&\begin{pmatrix} \tilde{m}^{[l+1]} \\ \tilde{w}^{[l+1]} \end{pmatrix} =
 \begin{pmatrix} m^{[l+1]} +\theta (m^{[l+1]}-m^{[l]}) \\ w^{[l+1]} +\theta (w^{[l+1]}-w^{[l]}) 
 \end{pmatrix},
\end{aligned}
\right.
\end{equation}
and, if $\gamma\tau<1$, the convergence of $(m^{[l]},w^{[l]})$ to a solution 
$(\hat{m},\hat{w})$ to \eqref{VARdis} is guaranteed together with the convergence of
$(n^{[l]},v^{[l]})$ to some
$(\hat{n},\hat{v})$ as $l\to\infty$. In order to compute the
Lagrange multiplier $\hat{u}\in \U$, which solves the first equation in
\eqref{MFGdis}, note that 
\eqref{e:lagrang} can be written equivalently as
\be\label{e:lagrang2}
\ba{rcl}
\L(m,w,u,\lambda)&:=& \varphi(m,w)- \langle (\lambda,u),\tilde{A}m + \tilde{B}w\rangle 
+ \langle \lambda, \bar{m} \rangle\\[6pt]
\; &=& \varphi(m,w)- \left\langle \begin{pmatrix}
	\tilde{A}^*\\
	\tilde{B}^*
\end{pmatrix}(\lambda,u),\begin{pmatrix}
m\\
w
\end{pmatrix}\right\rangle+\langle \lambda,\bar{m} \rangle,
\ea
\ee
and the 
optimality condition yields
\begin{equation}
	\label{e:optconddual}
\begin{pmatrix}
\tilde{A}^*\\
\tilde{B}^*
\end{pmatrix}\hat{z}
\in\partial \varphi(\hat{m},\hat{w}),
\end{equation}
where $(\hat{m},\hat{w})$ is the primal solution and $\hat{z}=(\hat{\lambda},\hat{u})$. 
Therefore, in order to 
approximate 
$\hat{z}$, note that from \eqref{eq_CP} we have
\begin{equation}
\begin{pmatrix}
\frac{m^{[l]} - m^{[l+1]}}{\tau} + \tilde{A}^*z^{[l+1]} \\  
\frac{w^{[l]} -w^{[l+1]}}{\tau} +\tilde{B}^*z^{[l+1]}
\end{pmatrix}\in\partial\varphi(m^{[l+1]},w^{[l+1]})
\end{equation}
and, hence, since the algorithm generates converging sequences $m^{[l]}\to\hat{m}$ and
$w^{[l]}\to\hat{w}$ and $z^{[l]}\to \hat{z}:=-Q^{-1}(\tilde{A}(\hat{n} + \gamma 
\hat{m})+\tilde{B}(\hat{v} + \gamma \hat{w})-\gamma(\bar{m},0))$, the closedness of the 
graph of $\partial\varphi$ \cite[Proposition~20.38]{MR3616647} yields 
\eqref{e:optconddual} and, hence, a good approximation of 
$\hat{z}$ is $z^{[l]}$ for $l$ large enough.
For obtaining $[u^*]_{i,j}^{N_T}$,
a good approximation is $[u^{[l]}]_{i,j}^{N_T}=g(x_{i,j},[m^{[l]}]^{N_T}_{i,j})$.
\begin{remark} 
{\rm(i)} In order to obtain an efficient algorithm, the computation of $\prox_{\tau \varphi}$  in \eqref{eq_CP} should be fast. A complete study of  $\prox_{\tau \varphi}$ is presented in \cite[Section 3.2]{BAKS} showing that its computation depends on the resolution of a real equation, which can be efficiently solved. \smallskip\\
{\rm(ii)}  An important step in \eqref{eq_CP} is the efficient computation of the inverse of 
$Q$. Different preconditioning strategies to tackle this issue will be presented in the 
following section. 
\end{remark}

\section{Preconditioning strategies} \label{sec_precond}
At the beginning of each iteration of the primal-dual algorithm \eqref{eq_CP}, we require the solution of a linear system 
\begin{equation}\label{matrix_Q_inversion}
Qz = b.
\end{equation}
The purpose of this section is to discuss preconditioning strategies for the solution of this linear system. For the stationary setting discussed in \cite{BAKS}, the solution of such a system via direct methods such as the \textsl{backlash} (\texttt{mldivide}) command in MATLAB \footnote{\texttt{http://uk.mathworks.com/help/matlab/ref/mldivide.html}} was feasible for relatively fine meshes (up to the order of 100 nodes per space dimension). However, as shown in  Table \ref{tab_backslash}, introducing a temporal dimension and thus increasing the degrees of freedom to $N_h^2\times N_T$ significantly increases the computation time. Indeed, the use of \textsl{backlash} on fine space and time grids -- e.g. $128^2$ space grid points and $40$ time steps -- requires an amount of RAM that is prohibitive on the machine used for our performance tests~\footnote{Intel Core i7-4600U @ 2.7GHz, 16GB RAM}, leading to ``out of memory'' errors. We mitigate this problem by exploring the solution of \eqref{matrix_Q_inversion} via preconditioned iterative methods, which perform efficiently for finer space and time subdivisions and different viscosities.
\begin{table}[!htb]
\begin{minipage}{.5\linewidth}
	\caption*{(a) $N_T = 10$}
	\centering
	\begin{tabular}{|c||*{4}{c|}}\hline
		\backslashbox{$\nu$}{$N_h$}
		&	\makebox[2.5em]{16}&\makebox[2.5em]{32}&\makebox[2.5em]{64}&\makebox[2.5em]{128}\\\hline\hline
		$5 \times 10^{-4}$ 		&	7.12	&	62.7	&	452		&	4720	 \\\hline
		$5 \times 10^{-3}$		&	6.29	&	60.6	&	345		&	3690	 \\\hline
		$5 \times 10^{-2}$ 		&	1.96	&	18.3	&	113		&	1340	 \\\hline
		$0.5$					&	1.18	&	9.41	&	56.1	&	660	 	 \\\hline
	\end{tabular}
\end{minipage}%
\begin{minipage}{.5\linewidth}
	\caption*{(b) $N_T = 40$}
	\centering
	\begin{tabular}{|c||*{4}{c|}}\hline
		\backslashbox{$\nu$}{$N_h$}
		&	\makebox[2.5em]{16}&\makebox[2.5em]{32}&\makebox[2.5em]{64}&\makebox[2.5em]{128}\\\hline\hline
		$5 \times 10^{-4}$ 	&	24.2	&	569		&	15600	&	{{[OOM]}}	 \\\hline
		$5 \times 10^{-3}$	&	18.8	&	496		&	14200	& 	{{[OOM]}}	 \\\hline
		$5 \times 10^{-2}$ 	&	8.10	&	145		&	5000	&	{{[OOM]}}	 \\\hline
		$0.5$				&	4.50	&	72.3	&	2510	&	{{[OOM]}}	 \\\hline
	\end{tabular}
\end{minipage}%
\vskip 2mm
	\caption{MATLAB's \textsl{backslash} computation times (seconds) for a single linear system solved in \eqref{eq_CP} within the Chambolle-Pock algorithm under a tolerance equal to $10^{-4}$ in in normalized $\ell^2$-norm. For fine meshes {[OOM]} indicates an out of memory error for the tested architecture. }
\label{tab_backslash}
\end{table}

\noindent We begin by illustrating the difficulties associated to the conditioning of the system in \eqref{matrix_Q_inversion}. Table \ref{conditions} shows the condition number of the system for different space-time discretizations and viscositiy values. Without any precoditioner, the condition numbers of different discretizations scale up to $10^8$. The same Table shows that by selecting a suitable preconditioner, such as the \textsl{modified incomplete Cholesky factorization} \cite{Benzi2002} (\texttt{michol} in MATLAB), the conditioning of the system is improved by 4 orders of magnitude.

We have tested different choices of preconditioners and iterative methods for our problem. Since the matrix $Q$ in our setting is sparse, symmetric, and  positive-definite, we have implemented an \textsl{incomplete Cholesky factorization} with diagonal scaling, a \textsl{modified incomplete Cholesky factorization}, and multigrid preconditioning. As for the choice of the iterative method, our tests included both \textsl{preconditioned conjugate gradient} (\texttt{pcg}), and the \textsl{biconjugate gradient stabilized method} (\texttt{BiCGStab}). The interested reader will find in \cite[Chapters 6 and 8]{wendland} a thorough description of the aforementioned methods, and in the Appendix of this article  performance tables for 
the different methods. 

Our findings suggest that the use of  an iterative \texttt{pcg} method, preconditioned by modified incomplete Cholesky factorization  
is satisfactory for small viscosities ($\nu\leq 0.05$).  However, this algorithm fails to converge for high viscosity systems on refined grids ($\nu = 0.5, \; N_T = 40, \; N_h \in \{64, 128\}$). Exchanging the \texttt{pcg} method by a \texttt{BiCGStab} algorithm preconditioned by modified incomplete Cholesky factorization slows down the process on finer grids, but allows for convergence in the failure cases of \texttt{pcg}: $\nu = 0.5, \; N_T = 40, \; N_h \in \{64, 128\}$.

In order to deal with (and exploit) the anisotropy of the system introduced by high viscosities, we have devised an algorithm consisting in a multigrid preconditioner with \texttt{BiCGStab} iterations akin to that described in Algorithm~\ref{precbcgstab}. It is the only among the tested methods which performs consistently for different viscosities and space-time discretizations. We discuss its implementation and assess its performance in the following section~\ref{ss_mg}.

\begin{table}[!htb]
	\caption{Condition numbers for $Q$ without preconditioning (a), and with modified incomplete Cholesky factorization preconditioning (b).}
	\begin{minipage}{.5\linewidth}
		\caption*{(a) No preconditioner (scaling $10^4$)}
		\centering
		\begin{tabular}{|c||*{3}{c|}}\hline
			\backslashbox[5em]{$\nu$}{$DoF$}
			&\makebox[3em]{$32^2\times 1$}&\makebox[3em]{$32^2\times 10$}&\makebox[3em]{$32^2\times 20$}\\\hline\hline
			$5 \times 10^{-5}$ 	&	4.290 	&	19.06	&	38.43	 \\\hline
			$5 \times 10^{-4}$ 	&	4.296 	&	19.25	&	39.10	 \\\hline
			$5 \times 10^{-3}$ 	&	4.751 	&	22.77	&	48.90	 \\\hline
			$5 \times 10^{-2}$ 	&	48.43   &	227.7   &	466.0	 \\\hline
			$0.5$ 				&	4399	&	19250    &	36540	 \\\hline
		\end{tabular}	
	\end{minipage}%
	\hfill
	\begin{minipage}{.5\linewidth}
		\caption*{(b) \texttt{michol} preconditioning (scaling $10^4$)}
		\centering
		\begin{tabular}{|c||*{3}{c|}}\hline
			\backslashbox[5em]{$\nu$}{$DoF$}
			&\makebox[3em]{$32^2\times 1$}&\makebox[3em]{$32^2\times 10$}&\makebox[3em]{$32^2\times 20$}\\\hline\hline
			$5 \times 10^{-5}$ 	&   0.04217		&	0.08535		&	0.1361	 \\\hline
			$5 \times 10^{-4}$ 	&	0.04218  	&	0.08612		&	0.1370	 \\\hline			
			$5 \times 10^{-3}$ 	&	0.04272		&	0.09325		&	0.1446	 \\\hline
			$5 \times 10^{-2}$ 	&	0.1025		&	0.2501		&	0.3579	 \\\hline
			$0.5$			 	&	1.255		&	2.743		&	3.770	 \\\hline
		\end{tabular}
	\end{minipage}%
	\label{conditions}
\end{table}

\begin{algorithm}
\caption{Preconditioned \texttt{BiCGStab}}\label{precbcgstab}
\begin{algorithmic}
\State $x_l \leftarrow $BiCGStab($Q_l, b_l, P_L,P_R,x_0,tol$) 
\Procedure {BiCGStab}{$Q_l, b_l, P_L,P_R,x_0,tol$}
\State $r_0:=p_0:=Q_lx_0-b_l;\quad \hat r_0:=\hat p_0:=P_lr_0;\quad \hat\rho_0:=\langle r_0,r_0\rangle;\quad k:=0$
\While{$\|r_k\|\geq tol$}
\State $v_k:=QP_R\hat p_k$
\State $\hat v_k:=P_Lv_k$
\State $\hat \alpha_k:=\hat\rho_k/\langle\hat v_k,\hat r_0\rangle$
\State $s_{k+1}:=r_k-\hat\alpha_kv_k$
\State $\hat s_{k+1}:=P_L s_{k+1}$
\State $t_{k+1}:=AP_R\hat s_{k+1}$
\State $\hat t_{k+1}:=P_Lt_{k+1}$
\State $\hat \omega_{k+1}:=\langle\hat s_{k+1},\hat t_{k+1}\rangle/\langle\hat t_{k+1},\hat t_{k+1}\rangle$
\State $\hat x_{k+1}:=\hat x_k+\hat\alpha_k\hat p_k+\hat\omega_{k+1}\hat s_{k+1}$
\State $r_{k+1}:=s_{k+1}-\hat\omega_{k+1}t_{k+1}$
\State $\hat r_{k+1}:=\hat s_{k+1}-\hat\omega_{k+1}\hat t_{k+1}$
\State $\hat \rho_{k+1}:=\langle\hat r_{k+1},\hat r_0\rangle$
\State $\hat \beta_{k+1}:=(\hat\alpha_{k}/\hat\omega_{k+1})(\rho_{k+1}/\rho_k)$
\State $\hat p_{k+1}:=\hat r_{k+1}+\hat\beta_{k+1}(\hat p_k-\hat\omega_{k+1}\hat v_k)$
\State $k:=k+1$
\EndWhile
\Return $x_k:=P_R\hat x_k$
\EndProcedure
\end{algorithmic}
\end{algorithm}

\subsection{Multigrid preconditioner}\label{ss_mg}
We implement  a multigrid preconditioned algorithm for solving \eqref{matrix_Q_inversion}. We refer the reader to~\cite{Multigrid2000} for an introduction and an overview of multigrid methods. We briefly review the main concepts behind the method. Consider two linear systems $A_1 \bar{x}_1=b_1$ and $A_0 \bar{x}_0=b_0$, stemming from two discretizations of a linear PDE over the grids $G_1$ and $G_0$, respectively. Assume also that $G_1$ is a refinement of $G_0$. Loosely speaking, the main idea of the method is that in order to find a good approximation of the solution $\bar{x}_1$ on the finer grid, we first consider what is known as a {\it smoothing} step. This step consists in 
 computing a few iterates $x_1^1, \; \hdots, x^{\eta_1}_1$ with  a standard indirect method, such as Jacobi or Gauss-Seidel, and to define the {\it residual} $r_1:=b_1 - A_1x^{\eta_1}_1$, which is shown to be {\it smoother} (less oscillatory) than the first residual $b_1-A_1 x^1$. Then, we consider in the coarser grid $G_0$ the second system $A_0 \bar{x}_0=b_0$ with $b_0=\hat{r}_1$, where $\hat{r}_1$ is the restriction of $r_1$ to $G_0$. Assuming that we can compute a good approximation of $\bar{x}_0$, which we still denote by $\bar{x}_0$, we then extend this solution to $G_1$ by using a linear interpolation.  Calling $e_1$ the resulting vector, we update $x^{\eta_1}_1$ by redefining   it as $x^{\eta_1}_1+e_1$ and we end the procedure by applying again a few iterations, say $\eta_2$, of a smoothing method initialized at  $x^{\eta_1}_1$. This last step is called {\it post smoothing}.

The previous paragraph introduced what is known as a {\it two grid iteration}. If we consider more grids $G_0$, $G_1$,$\hdots$, $G_\ell$, where for each $k=0, \hdots, \ell-1$, $G_{k} \subseteq G_{k+1}$, we can proceed similarly and obtain a better approximation of the solution to $A_\ell \bar{x}_\ell=b_\ell$.  As in the previous case, we begin with the finest grid $G_\ell$ and we perform $\eta_1$ smoothing steps to obtain the residual $r_\ell:=b_\ell - A_\ell x^{\eta_1}_{\ell}$ whose restriction to $G_{\ell-1}$ is denoted by $\hat{r}_\ell$. In this grid we consider the system $A_{\ell-1} x_{\ell-1}= \hat{r}_\ell$ and we perform again a smoothing step and a restriction of the residual to $G_{\ell-2}$. The procedure continues until we get to the coarsest grid $G_0$, where the solution $e_0$ to the corresponding linear system can be found easily (typically using a direct method). Next, the solution $e_1$ on the grid $G_1$ is corrected with the interpolation of $e_0$. Another post smoothing is performed to the corrected solution on $G_{1}$ and using its interpolation in the grid $G_{2}$ we correct the previous solution on this grid. The smoothing, interpolation and correction iterations end  once we arrive to the finest grid $G_\ell$ to obtain the final approximation of $\bar{x}_\ell$. The previous procedure is called a multigrid method with a {\it $V$-cycle}. An alternative, to obtain a more accurate solution, is to proceed as before going from $G_\ell$ to $G_{\ell-1}$ and then for $k=\ell-1,\hdots, 1$ to perform two consecutive coarse-grid corrections, instead of one as in the $V$-cycle. The resulting procedure is known as multigrid with a {\it $W$-cycle}.  Finally, in between the $V$-cycle and the $W$-cycle, we have the {\it F-cycle}, where in the process of going from the coarsest grid to the finest one, if a grid has been reached for the first time, another correction with the coarser grids using a $V$-cycle is performed. 

In our context, we use one cycle of the multigrid algorithm, which is a linear operator as a function of the residual on the finest grid, as a preconditioner for solving \eqref{matrix_Q_inversion} with the \texttt{BiCGStab} method. Since $Q$ is related to the finite difference discretization of the operator $-\partial_{tt}^2 + \nu^2 \Delta^2- \Delta$ and $\nu$ is not necessarily small, as in \cite{AchdouPerez2012}, it is natural to consider the refinements of the grid only in the space variable (we refer the reader to \cite{Multigrid2000} for  semi-coarsening multigrid methods in the context of  anisotropic operators). We suppose that the spatial mesh is such that $N_h=H2^{\ell}$, with $H>1$ and $\ell$ is a positive integer (in the numerical example in the next section $H$ will be equal to $2$ or $3$, $H^2$ being the number of spatial points in the coarsest grid).

Let us specify the main steps of the multigrid method we use as a preconditioner.

\begin{itemize}
\item[$\diamond$] {\it Hierachy of Grids:} Semi-coarsened grids $G_k$ with size $(N_T+1)H^22^{2k}$ for all $k=0\dots \ell$. \vspace{0.2cm}
\item[$\diamond$]{\it Cycle:} We use the F-cycle.\vspace{0.2cm}
\item[$\diamond$] {\it Restriction operator:} As in \cite{AchdouPerez2012}, in order to restrict the residual on the grid $G_k$ to the grid $G_{k-1}$, we use the second-order operator $R_k:\RR^{(2^kH)^2(N_T+1)}\rightarrow \RR^{(2^{k-1}H)^2(N_T+1)}$ defined by
$$
(R_kX)^n_{i,j} := \frac1{16}
\left(
\begin{aligned}
4X^n_{2i,2j} + 2(X^n_{2i+1,2j}+X^n_{2i-1,2j}+X^n_{2i,2j+1}+X^n_{2i,2j-1}) \\
X^n_{2i-1,2j-1}+X^n_{2i-1,2j+1}+X^n_{2i+1,2j-1}+X^n_{2i+1,2j+1}
\end{aligned}
\right),
$$
for $n=0,\hdots, N_{T}$, $i$, $j=1, \hdots, 2^{k-1}H$.
\vspace{0.1cm}
\item[$\diamond$] {\it Interpolation operator:} We denote by $I_k:\RR^{(2^{k-1}H)^2(N_T+1)}\rightarrow \RR^{(2^kH)^2(N_T+1)}$ the interpolation operator from the grid $G_{k-1}$ to the grid $G_k$. We have chosen a standard bilinear interpolation operator in the space variable, which is also a second-order operator  and dual to the restriction operator ($I_k = 4R_k^*$).  According to  \cite{Brandt_1994}, the sum of the orders of $R_k$ and $I_k$ has to be at least equal to the degree of the differential operator. In our context, both are equal to $4$.
 \vspace{0.2cm}
\item[$\diamond$] {\it Linear systems on the different grids:} The linear systems are defined by the matrices $$Q_k:=A_kA_k^*+B_kB_k^*, \hspace{0.5cm} k=0, \hdots, \ell,$$
where we recall that $A_k$ and $B_k$ are the finite difference discretizations of $\partial_{t} - \nu \Delta$ and $\mbox{div}(\cdot)$, respectively, on the grid $G_k$ (see \eqref{definition_of_A_B}). \vspace{0.2cm}
%
\item[$\diamond$] {\it Smoother:} Here we have used Gauss-Seidel iterations in the lexicographic order. There is no reason for choosing the lexicographic order, other than its simplicity.
%
\vspace{0.2cm}

\item[$\diamond$] {\it Solving the system on the coarsest grid $G_0$:} We can use an exact solver such as \textsl{backlash} in MATLAB. Indeed, in $G_0$ the size of the system is really small with respect to the size of the system on the grid $G_\ell$ (in $G_0$, we can even store the inverse of  $Q_0$  and inversion at this level just becomes a matrix multiplication).
\end{itemize}\vspace{0.2cm}

The multigrid preconditoning procedure is summarized in Algorithm \ref{mgmfg}.

\begin{algorithm} 
\caption{Multigrid Preconditioner for $Q_{\ell}x_{\ell} = b_{\ell}$}\label{mgmfg}
\begin{algorithmic}
\State $P_L: y \mapsto$ MultigridSolver($\ell, 0, y,$cycle)
\State $x_l \leftarrow $BiCGStab($Q_{\ell}, b_{\ell},P_L, I_d, x_0,$ \emph{tol}) 
\Procedure {MultigridSolver}{$k,x_k,b_k,$cycle}
\If{$k = 0$}
\State $x_k \leftarrow Q_0^{-1}b_k$
\Else
\State $x_k \leftarrow $Perform $\eta_1$ steps of Gauss-Seidel from $x_k$ with $b_k$ as second member.
\State $x_{k-1} \leftarrow 0$
\State $x_{k-1} \leftarrow $MultigridSolver($k-1, x_{k-1}, R_k(b_k-Q_kx_k),$cycle)
\If{cycle is W}
\State $x_{k-1} \leftarrow $MultigridSolver($k-1, x_{k-1}, R_k(b_k-Q_kx_k),$cycle)
\EndIf
\If{cycle is F}
\State $x_{k-1} \leftarrow $MultigridSolver($k-1, x_{k-1}, R_k(b_k-Q_kx_k),$V)
\EndIf
\State $x_k \leftarrow x_k + I_kx_{k-1}$
\State $x_k \leftarrow $Perform $\eta_2$ steps of Gauss-Seidel from $x_k$ with $b_k$ as second member.
\EndIf
\Return $x_k$
\EndProcedure
\end{algorithmic}
\end{algorithm}

\subsection{Numerical Tests}
In this section we present a test case considered in~\cite{MR2679575}, for which the stationary solution has been computed numerically in~\cite{BAKS} using the primal-dual algorithm presented above.

The setting is as follows: we consider system \eqref{MFGcont} with  $g \equiv 0$ and 
$$
	f(x,y,m) := m^2 - \overline H(x,y), \qquad \overline H(x,y) = \sin(2 \pi y) + \sin(2 \pi x) + \cos(2 \pi x),
$$
for all  $(x,y) \in \RR^2$ and $m \in \RR_+$. This means that in the underlying differential game modelled by \eqref{MFGcont}, a typical agent aims to get closer to the maxima of $\bar{H}$ and, at the same time, he/she is adverse to crowded regions (because of the presence of the $m^2$ term in $f$).
%
%

We first validate the dynamic behavior of our solution. Figure~\ref{fig:evolm} shows the evolution of the mass at four different time steps. Starting from a constant initial density, the mass converges to a steady state, and then, when $t$ gets close to the final time $T$, the mass is influenced by the final cost and converges to a final state. This behavior is referred to as \emph{turnpike phenomenon} in the literature \cite{Porretta2016}. It is illustrated by Figure~\ref{fig:dist-statio}, which displays as a function of time $t$ the distance of the mass at time $t$ to the stationary state computed as in~\cite{BAKS}. 
In other words, denoting by $m^\infty \in \mathbb{R}^{N_h\times N_h}$ the solution to the discrete stationary problem and by $m \in \M$ the solution to the discrete evolutive problem, Figure~\ref{fig:dist-statio} displays the graph of 
$k \mapsto \|m_\infty - m^k\|_{\ell_2} = \left(h^2 \sum_{i,j} (m^\infty_{i,j} - m^k_{i,j})^2 \right)^{1/2}$, $k \in \{0,\dots,N_T\}$.

\begin{figure}	
	\centering
	\begin{subfigure}[t]{0.45\linewidth}
		\centering
		\includegraphics[width=\linewidth]{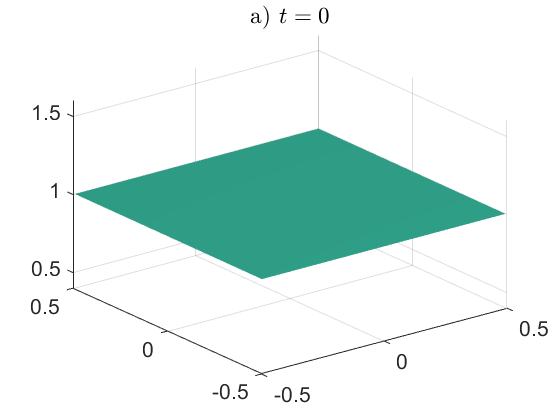}
	\end{subfigure}
	\quad
	\begin{subfigure}[t]{0.45\linewidth}
		\centering
	  	\includegraphics[width=\linewidth]{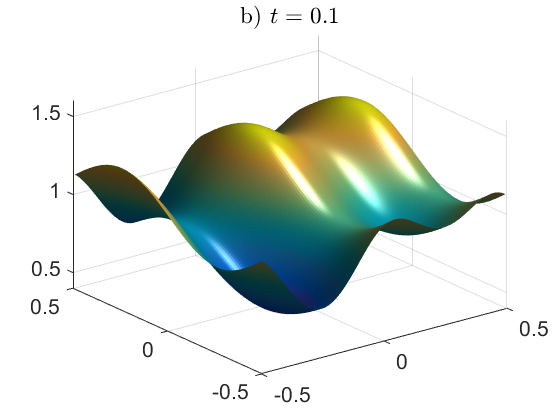}
	\end{subfigure}
	
	\begin{subfigure}[t]{0.45\linewidth}
		\centering
	  	\includegraphics[width=\linewidth]{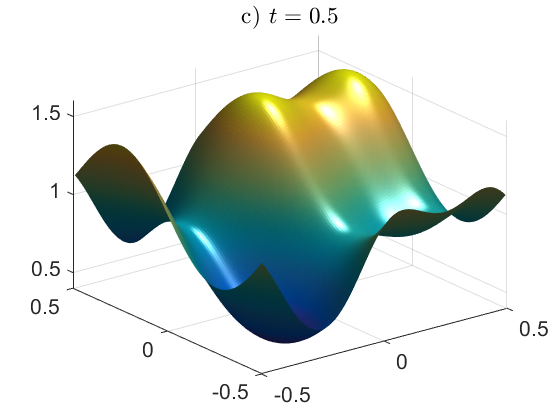}	
	\end{subfigure}
	\quad
	\begin{subfigure}[t]{0.45\linewidth}
		\centering
	  	\includegraphics[width=\linewidth]{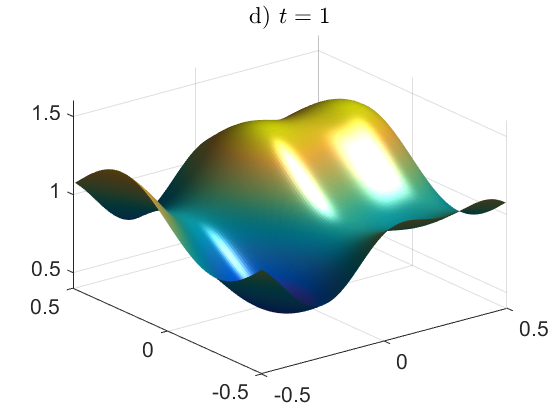}
	\end{subfigure}
	\caption{\label{fig:evolm} Evolution of the density $m$ obtained with the multi-grid preconditioner for $\nu = 0.5, T=1, N_T = 200$ and $N_h = 128$. At $t=0.12$ the solution is close to the solution of the associated  stationary MFG.}
\end{figure}

\begin{figure}	
		\centering
		\includegraphics[width=0.6\linewidth]{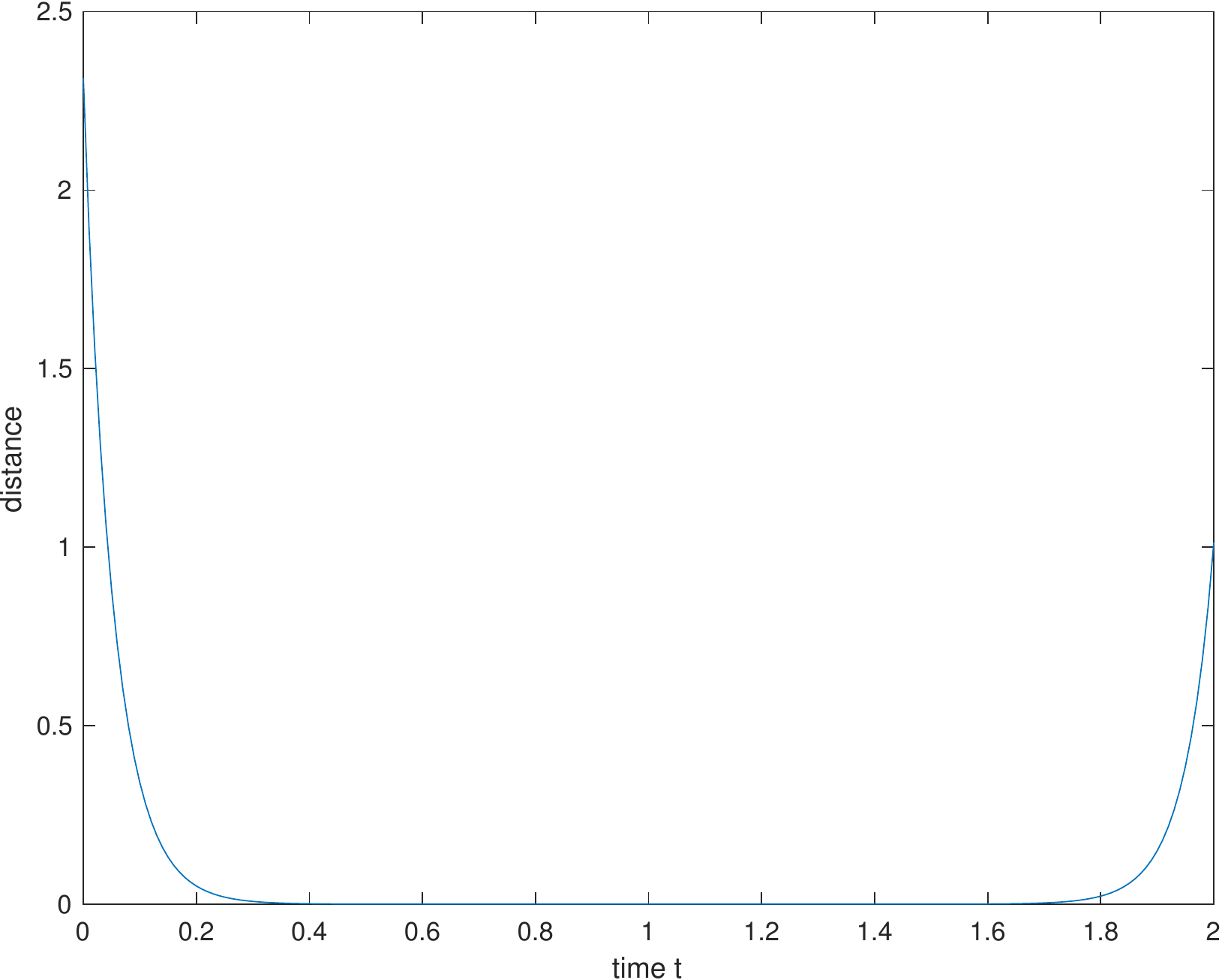}	
	\caption{\label{fig:dist-statio} Distance to the stationary solution at each time $t \in [0,T]$, for $\nu = 0.5, T=2, N_T = 200$ and $N_h =128$. The distance is computed using the $\ell^2$ norm  as explained in the text. The turnpike phenomenon is observed as for a long time frame the time-dependent mass approaches the solution of the stationary MFG.}
\end{figure}

For the multi-grid preconditioner, Table \ref{tab:MG_cvg_time} shows the computation times for different discretizations. It can be observed that finer meshes with $128^3$ degrees of freedom are solvable within CPU times which outperfom others methods shown in the Appendix and in \cite{BAKS}. Furthermore, the method is robust with respect to different viscosities. 

From Table \ref{tab:MG_cvg_time} we observe that most of the computational time is used for solving the second proximal operator (the third equality of \eqref{eq_CP}), which does not use a multigrid strategy but which is a pointwise operator (see Proposition $3.1$ of \cite{BAKS}) and thus could be fully paralellizable.

\begin{table}[!htb]
    \begin{minipage}{.5\linewidth}
      \caption*{(a) Grid with $64 \times 64 \times 64$ points.}
      \centering
        \begin{tabular}{| c | c | c | c |}
            \hline
            $\nu$ & Total time & Time first prox & Iterations \\
            \hline
            $0.6$ & $116.3$ [s] & $11.50$ [s] & $20$ \\
            \hline
            $0.36$ & $120.4$ [s] & $11.40$ [s] & $21$  \\
            \hline
            $0.2$  & $119.0$ [s]  & $11.26$ [s] & $22$ \\
            \hline
            $0.12$  &  $129.1$ [s] & $14.11$ [s] & $22$  \\
            \hline
            $0.046$ &  $225.0$ [s] & $23.28$ [s] & $39$  \\
            \hline
        \end{tabular}
    \end{minipage}%
    \begin{minipage}{.5\linewidth}
      \centering
        \caption*{(b) Grid with $128 \times 128 \times 128$ points.}
        \begin{tabular}{| c | c | c | c |}
            \hline
            $\nu$ & Total time & Time first prox & Iterations \\
            \hline
            $0.6$ &  $921.1$ [s] & $107.2$ [s] &  $20$ \\
            \hline
            $0.36$ &  $952.3$ [s] &  $118.0$ [s] & $21$ \\
            \hline
            $0.2$  &  $1028.8$ [s] &  $127.6$ [s] & $22$ \\
            \hline
            $0.12$  & $1036.4$ [s]  &  $135.5$ [s] & $23$ \\
            \hline
            $0.046$ &  $1982.2$ [s] & $260.0$ [s] & $42$ \\
            \hline
        \end{tabular}
    \end{minipage} 
  \vskip 2mm  	
    \caption{Time (in seconds) for the convergence of the Chambolle-Pock algorithm, cumulative time of the first proximal operator with the multigrid preconditioner, and number of iterations, for different viscoty values $\nu$ and two types of grids. Here we used $\eta_1 = \eta_2 = 2, T=1$ and a tolerance between two iterations of the Chambolle-Pock algorithm equal to  $10^{-6}$ in normalized $\ell^2$-norm.}\label{tab:MG_cvg_time}
\end{table}

Unlike the stationary case, low viscosities seem to make the algorithm be slightly slower. However, Table \ref{tab:MG_bicg_ite} shows that the average number of iterations of \texttt{BiCGStab} stays low regardless of the viscosity. Indeed Table \ref{tab:MG_cvg_time} shows that more Chambolle-Pock iterations are needed to converge. The same behavior happens when we use a direct exact solver instead of the multi-grid preconditioned \texttt{BiCGStab} algorithm.

\begin{table}[!htb]

    \begin{minipage}{.5\linewidth}
      \caption*{(a) iterations to decrease the residual by a factor $10^{-3}$.}
      \centering
        \begin{tabular}{| c | c | c | c |}
            \hline
            $\nu$ & {\tiny $32\times32\times32$} & {\tiny $64\times64\times64$} & {\tiny $128\times128\times128$}\\
            \hline
            $0.6$ & $1.65$ & $1.86$ &  $2.33$\\
            \hline
            $0.36$ & $1.62$ & $1.90$ &  $2.43$ \\
            \hline
            $0.2$  & $1.68$ & $1.93$ & $2.59$ \\
            \hline
            $0.12$ & $1.84$ & $2.25$ &  $2.65$ \\
            \hline
            $0.046$ & $1.68$ & $2.05$ &  $2.63$ \\
            \hline
        \end{tabular}
    \end{minipage}%
    \begin{minipage}{.5\linewidth}
      \centering
        \caption*{(b) iterations to solve the system with an error of $10^{-8}$.}
        \begin{tabular}{| c | c | c | c |}
            \hline
            $\nu$ & {\tiny $32\times32\times32$} & {\tiny $64\times64\times64$} & {\tiny $128\times128\times128$}\\
            \hline
            $0.6$ & $3.33$ & $3.40$ &  $3.38$\\
            \hline
            $0.36$ & $3.10$ & $3.21$ &  $3.83$ \\
            \hline
            $0.2$  & $3.07$ & $3.31$ & $4.20$ \\
            \hline
            $0.12$ & $3.25$ & $3.73$ &  $4.64$ \\
            \hline
            $0.046$ & $2.88$ & $3.59$ &  $4.67$ \\
            \hline
        \end{tabular}
    \end{minipage} 
    	
    \caption{Average number of iterations of the preconditioned \texttt{BiCGStab} with $\eta_1 = \eta_2 = 2, T=1$ and a tolerance between two iterations of the Chambolle-Pock algorithm equal to $10^{-6}$ in normalized $\ell^2$-norm.}   \label{tab:MG_bicg_ite}
\end{table}

{\bf Concluding Remarks.} In this work we have developed a first-order primal-dual algorithm for the solution of second-order, time-dependent mean field games. The procedure consists of: a variational formulation for the MFG, its discretization via finite differences, the application Chambolle-Pock algorithm to the resulting minimization. While this method has been studied  for stationary MFG in \cite{BAKS}, its numerical realization for time-dependent MFGs was prohibitive in terms of computing time, as the Chambolle-Pock iteration requires the solution of a large-scale linear system at each iteration. We have overcome this difficulty by studying different preconditioning strategies for the associated linear system. Overall, the multigrid preconditioner with a \texttt{BiCGStab} iteration performs satisfactorily for different discretizations and viscosity values. \medskip\\

{\bf Acknowledgments.}  The third author wants to acknowledge funding within the ANR project MFG ANR-16-CE40-0015-01 operated by the French National Research Agency (ANR).

Most of this work was realized while the fourth author was a postdoctoral fellow at NYU Shanghai and supported by a discretionary research fund.

During the first phase of the project, the fifth author was affiliated to the Unit\'e de Math\'ematiques Pures et Appliqu\'ees (UMPA) UMR CNRS $5669$, of the \'Ecole Normale Sup\'erieure de Lyon and to the Project-Team Beagle of the Inria Rh\^one-Alpes. He wishes to acknowledge funding within the framework of the LABEX MILYON (ANR-10-LABX-0070) of the Universit\'e de Lyon, within the program "Investissements d'Avenir" (ANR-11-IDEX-0007). In addition, his participation to this project has been partially supported by the European Research Council (ERC) under the European Union's Horizon 2020 research and innovation program (grant agreement No. 639638). His participation  to this project has also been partially supported by a CEMRACS 2017 scholarship.

The sixth author thanks the  support from the PGMO project VarPDEMFG and from the ANR project MFG ANR-16-CE40-0015-01. 

\bibliographystyle{plain}
\bibliography{biblio_MFG_CEMRACS}

\begin{thebibliography}{10}

\bibitem{MR2888257}
Y.~Achdou, F.~Camilli, and I.~Capuzzo-Dolcetta.
\newblock Mean field games: numerical methods for the planning problem.
\newblock {\em SIAM J. Control Optim.}, 50(1):77--109, 2012.

\bibitem{MR3097034}
Y.~Achdou, F.~Camilli, and I.~Capuzzo-Dolcetta.
\newblock Mean field games: convergence of a finite difference method.
\newblock {\em SIAM J. Numer. Anal.}, 51(5):2585--2612, 2013.

\bibitem{MR2679575}
Y.~Achdou and I.~Capuzzo-Dolcetta.
\newblock Mean field games: numerical methods.
\newblock {\em SIAM J. Numer. Anal.}, 48(3):1136--1162, 2010.

\bibitem{AchdouPerez2012}
Y.~Achdou and V.~Perez.
\newblock Iterative strategies for solving linearized discrete mean field games
  systems.
\newblock {\em Netw. Heterog. Media}, 7(2):197--217, 2012.

\bibitem{MR3452251}
Y.~Achdou and A.~Porretta.
\newblock Convergence of a finite difference scheme to weak solutions of the
  system of partial differential equations arising in mean field games.
\newblock {\em SIAM J. Numer. Anal.}, 54(1):161--186, 2016.

\bibitem{ACFK17}
G.~Albi, Young-Pil Choi, M.~Fornasier, and D.~Kalise.
\newblock Mean-field control hierarchy.
\newblock {\em Appl. Math. Optim.}, 76(1):93--175, 2017.

\bibitem{andreev_17}
R.~Andreev.
\newblock Preconditioning the augmented {L}agrangian method for instationary
  mean field games with diffusion.
\newblock {\em SIAM J. Sci. Comput.}, 39(6):A2763--A2783, 2017.

\bibitem{MR3616647}
H.~H. Bauschke and P.-L. Combettes.
\newblock {\em Convex analysis and monotone operator theory in {H}ilbert
  spaces}.
\newblock CMS Books in Mathematics/Ouvrages de Math\'ematiques de la SMC.
  Springer, Cham, second edition, 2017.

\bibitem{MR3395203}
J.-D. Benamou and G.~Carlier.
\newblock Augmented {L}agrangian methods for transport optimization, mean field
  games and degenerate elliptic equations.
\newblock {\em J. Optim. Theory Appl.}, 167(1):1--26, 2015.

\bibitem{MR3134900}
A.~Bensoussan, J.~Frehse, and P.~Yam.
\newblock {\em Mean field games and mean field type control theory}.
\newblock SpringerBriefs in Mathematics. Springer, New York, 2013.

\bibitem{Benzi2002}
M.~Benzi.
\newblock Preconditioning techniques for large linear systems: A survey.
\newblock {\em J. Comput. Phys.}, 182(2):418--477, nov 2002.

\bibitem{Brandt_1994}
A.~Brandt.
\newblock Rigorous quantitative analysis of multigrid. {I}. {C}onstant
  coefficients two-level cycle with {$L_2$}-norm.
\newblock {\em SIAM J. Numer. Anal.}, 31(6):1695--1730, 1994.

\bibitem{M+S}
L.~M. {Brice\~{n}o-Arias} and P.-L Combettes.
\newblock A monotone+ skew splitting model for composite monotone inclusions in
  duality.
\newblock {\em SIAM J. Optim.}, 21(4):1230--1250, 2011.

\bibitem{BAKS}
L.~M. {Brice\~{n}o-Arias}, D.~Kalise, and F.~J. Silva.
\newblock Proximal methods for stationary mean field games with local
  couplings.
\newblock {\em SIAM J. Control Optim.}, 56(2):801--836, 2018.

\bibitem{BFMW14}
M.~Burger, M.~Di~Francesco, P.~A. Markowich, and M.~T. Wolfram.
\newblock Mean field games with nonlinear mobilities in pedestrian dynamics.
\newblock {\em Discrete Contin. Dyn. Syst. Ser. B}, 19(5):1311--1333, 2014.

\bibitem{Cardialaguet10}
P.~Cardaliaguet.
\newblock Notes on {M}ean {F}ield {G}ames: from {P.-L.} {L}ions' lectures at
  {C}oll\`ege de {F}rance.
\newblock {\em Lecture {N}otes given at {T}or {V}ergata}, 2010.

\bibitem{MR3408214}
P.~Cardaliaguet.
\newblock Weak solutions for first order mean field games with local coupling.
\newblock In {\em Analysis and geometry in control theory and its
  applications}, volume~11 of {\em Springer INdAM Ser.}, pages 111--158.
  Springer, Cham, 2015.

\bibitem{MR3358627}
P.~Cardaliaguet and P.~J. Graber.
\newblock Mean field games systems of first order.
\newblock {\em ESAIM Control Optim. Calc. Var.}, 21(3):690--722, 2015.

\bibitem{MR3399179}
P.~Cardaliaguet, P.~J. Graber, A.~Porretta, and D.~Tonon.
\newblock Second order mean field games with degenerate diffusion and local
  coupling.
\newblock {\em NoDEA Nonlinear Differential Equations Appl.}, 22(5):1287--1317,
  2015.

\bibitem{MR2782122}
A.~Chambolle and T.~Pock.
\newblock A first-order primal-dual algorithm for convex problems with
  applications to imaging.
\newblock {\em J. Math. Imaging Vision}, 40(1):120--145, 2011.

\bibitem{ChenT}
G.~{Chen} and M.~Teboulle.
\newblock A proximal-based decomposition method for convex minimization
  problems.
\newblock {\em Math. Program.}, 64:81--101, 1994.

\bibitem{ADMM}
D~{Gabay} and B~Mercier.
\newblock A dual algorithm for the solution of nonlinear variational problems
  via finite element approximation.
\newblock {\em Comput. Math. Appl.}, 2(1):17--40, 1976.

\bibitem{ADMM2}
R.~{Glowinski } and A.~Marrocco.
\newblock Sur l'approximation, par \'el\'ements finis d'ordre un, et la
  r\'esolution, par p\'enalisation-dualit\'e, d'une classe de probl\`emes de
  {D}irichlet non lin\'eaires.
\newblock {\em Rev. Fran\c caise Automat. Informat. Recherche Op\'erationnelle
  S\'er. Rouge Anal. Num\'er.}, 9(R-2):41--76, 1975.

\bibitem{MR3559742}
D.-A. Gomes, E.~A. Pimentel, and V.~Voskanyan.
\newblock {\em Regularity theory for mean-field game systems}.
\newblock SpringerBriefs in Mathematics. Springer, Cham, 2016.

\bibitem{MR3195844}
D.-A. Gomes and J.~Sa\'ude.
\newblock Mean field games models - a brief survey.
\newblock {\em Dyn. Games Appl.}, 4(2):110--154, 2014.

\bibitem{MR2647032}
A.~Lachapelle, J.~Salomon, and G.~Turinici.
\newblock Computation of mean field equilibria in economics.
\newblock {\em Math. Models Methods Appl. Sci.}, 20(4):567--588, 2010.

\bibitem{LasryLions06ii}
J.-M. Lasry and P.-L. Lions.
\newblock Jeux \`a champ moyen {II}. {H}orizon fini et contr\^ole optimal.
\newblock {\em C. R. Math. Acad. Sci. Paris}, 343:679--684, 2006.

\bibitem{LasryLions07}
J.-M. Lasry and P.-L. Lions.
\newblock Mean field games.
\newblock {\em Jpn. J. Math.}, 2:229--260, 2007.

\bibitem{MR3420414}
A.~R. M\'esz\'aros and F.~J. Silva.
\newblock A variational approach to second order mean field games with density
  constraints: the stationary case.
\newblock {\em J. Math. Pures Appl. (9)}, 104(6):1135--1159, 2015.

\bibitem{meszaros_silva_preprint_17}
A.~R. M\'esz\'aros and F.~J. Silva.
\newblock On the variational formulation of some stationary second order mean
  field games systems.
\newblock {\em SIAM J. Mathematical Analysis}, 50(1):1255--1277, 2018.

\bibitem{MR3158785}
N.~Papadakis, G.~Peyr\'e, and E.~Oudet.
\newblock Optimal transport with proximal splitting.
\newblock {\em SIAM J. Imaging Sci.}, 7(1):212--238, 2014.

\bibitem{MR3305653}
A.~Porretta.
\newblock Weak solutions to {F}okker-{P}lanck equations and mean field games.
\newblock {\em Arch. Ration. Mech. Anal.}, 216(1):1--62, 2015.

\bibitem{Porretta2016}
A.~Porretta and E.~Zuazua.
\newblock Remarks on long time versus steady state optimal control.
\newblock In {\em Mathematical Paradigms of Climate Science}, pages 67--89.
  Springer, Cham, 2016.

\bibitem{MR0211759}
R.~T. Rockafellar.
\newblock Duality and stability in extremum problems involving convex
  functions.
\newblock {\em Pacific J. Math.}, 21:167--187, 1967.

\bibitem{Multigrid2000}
U.~Trottenberg, C.~W. Oosterlee, and A.~Schuller.
\newblock {\em Multigrid}.
\newblock Academic Press, 2000.

\bibitem{wendland}
H.~Wendland.
\newblock {\em Numerical Linear Algebra}.
\newblock Cambridge Texts in Applied Mathematics. Cambridge University Press,
  Cambridge, United Kindgom, first edition, 2017.

\end{thebibliography}
\section*{Appendix}

\begin{table}[!htb]
\begin{minipage}{.48\linewidth}
	\caption*{(a) Unpreconditioned}
	\centering
	\begin{tabular}{|c||*{4}{c|}}\hline
		\backslashbox{$\nu$}{$N_h$}
		&\makebox[2.5em]{16}&\makebox[2.5em]{32}&\makebox[2.5em]{64}&\makebox[2.5em]{128}\\\hline\hline
		$5 \times 10^{-4}$ 	& 	18,5	&	87,6		&	448		&	1720	 \\\hline
		$5 \times 10^{-3}$ 	&	22,1	&	98,6	&	392		&	1750	 \\\hline
		$5 \times 10^{-2}$ 	&	20,7	&	93,4	&	607	&	8240		 \\\hline
		$0.5$					&	24,9	&	113	&	
			 {{[X]}	}&	
			 	 {{[X]}}	 \\\hline
	\end{tabular}
\end{minipage}
\hfill
\begin{minipage}{.48\linewidth}
	\caption*{(b) \texttt{michol}}
	\centering
	\begin{tabular}{|c||*{4}{c|}}\hline
		\backslashbox{$\nu$}{$N_h$}
		&\makebox[2.5em]{16}&\makebox[2.5em]{32}&\makebox[2.5em]{64}&\makebox[2.5em]{128}\\\hline\hline
		$5 \times 10^{-4}$ 		&	15,8	&	77,8	&	346	&	1390		 \\\hline
		$5 \times 10^{-3}$ 		&	12,4	&	80,9	&	325	&	1244		 \\\hline
		$5 \times 10^{-2}$ 		&	5,47	&	26,3	&	138	&	636		 \\\hline
		$0.5$					&	3,69	&	16,5	&	
			 {{[X]}}	& {{[X]}}	 \\\hline
	\end{tabular}
\end{minipage}%
\vskip 2mm
	\caption{Conjugate Gradient computation times (s). (a) Unpreconditioned. (b) Preconditioned with modified incomplete Cholesky factorization. Time discretization: $N_T=40$. {{[X]}} indicates no convergence.}
\label{tab_PCG}
\end{table}

\begin{table}[!htb]

\begin{minipage}{.48\linewidth}
	\caption*{(a) Unpreconditioned}
	\centering
	\begin{tabular}{|c||*{5}{c|}}\hline
		\backslashbox{$\nu$}{$N_h$}
		&\makebox[2.5em]{8}&\makebox[2.5em]{16}&\makebox[2.5em]{32}&\makebox[2.5em]{64}&\makebox[2.5em]{128}\\\hline\hline
		$5 \times 10^{-4}$ 	&	2.42	&	14.2 &	69.0	&	294 &	1210 \\\hline
		$5 \times 10^{-3}$ 	&	3.09	&	16.3	&	63.9	&	270		&	1210	\\\hline
		$5 \times 10^{-2}$ 	&	1.41	&	12.9	&	61.3	&	389		&	5470	 \\\hline
		$0.5$					&	3.41	&	16.5	&	98.9	&	{{[X]}}	&	{{[X]}}	 \\\hline
	\end{tabular}
\end{minipage}%
\hfill
\begin{minipage}{.48\linewidth}
	\caption*{(b) \texttt{michol}}
	\centering
	\begin{tabular}{|c||*{4}{c|}}\hline
		\backslashbox{$\nu$}{$N_h$}
		&\makebox[2.5em]{16}&\makebox[2.5em]{32}&\makebox[2.5em]{64}&\makebox[2.5em]{128}\\\hline\hline
		$5 \times 10^{-4}$ 	&	15.7	&	80.4	&	412	&	1890	 \\\hline
		$5 \times 10^{-3}$ 	&	12.2	&	82.2	&	369	&	1650	 \\\hline
		$5 \times 10^{-2}$ 	&	5.25	&	27.2	&	174	&	894	 \\\hline
		$0.5$					&	3.53	&	18.8	&	122	&	2120	 \\\hline
	\end{tabular}
\end{minipage}
\vskip 2mm
	\caption{\texttt{BiCGStab} computation times (s). (a) Unpreconditioned. (b) Preconditioned with modified incomplete Cholesky factorization. Time discretization: $N_T=40$. {{[X]}} indicates no convergence.}
\label{tab_BiCGStab}
\end{table}

\end{document}